\documentclass[small]{article}

\usepackage{amsmath,amstext,amsthm,amsfonts}
\usepackage{makeidx}

\usepackage{marginnote}
\usepackage{bbding}

\usepackage{wasysym}
\usepackage{amssymb,latexsym}
\usepackage{txfonts, pxfonts}
\usepackage[english]{babel}
\usepackage[pdftex]{graphicx}
\usepackage{xcolor}
\numberwithin{equation}{section}
\newtheorem{theorem}{\textbf{Theorem}}[section]
\newtheorem{proposition}[theorem]{\textbf{Proposition}}

\newtheorem{lemma}[theorem]{\textbf{Lemma}}

\theoremstyle{definition}

\theoremstyle{remark}
\newtheorem{remark}[theorem]{\it{Remark}}

\def\e{\epsilon}

\def\R{\mathbb{R}}
\def\Rn{{\mathbb{R}}^n_+}

\def\d{\partial}

\def\fermi{\psi_i:B^+_{\delta}(0)\to M}
\def\fermilinha{\psi_i:B^+_{\delta'}(0)\to M}
\def\a{\alpha}

\def\ba{\begin{align}}
\def\ea{\end{align}}
\def\bp{\begin{proof}}
\def\ep{\end{proof}}

\renewcommand{\(}{\left(}
\renewcommand{\)}{\right)}

\begin{document}

\title{A positive mass theorem for non-smooth metrics on asymptotically flat manifolds with non-compact boundary}
\date{}

\author{\textsc{S\'ergio Almaraz}\footnote{Partially supported by grant 201.049/2022, FAPERJ/Brazil.}  \textsc{and Shaodong Wang}\footnote{Partially supported by NSFC 12001364 and the Fundamental Research Funds for the Central Universities, No.30924010839.}}

\maketitle
%\tableofcontents

\begin{abstract}
On a smooth asymptotically flat Riemannian manifold with non-compact  boundary, we prove a positive mass theorem for metrics which are only continuous across a compact hypersurface. As an application, we obtain a positive mass theorem on manifolds with non-compact corners. 

\end{abstract}



%\tableofcontents

%%%%%%%%%%%%%%%%%%%%%%%%%%%%%%%%%%%%%%%%%%%%%%%%%%%%%%%%%%%%%%%%%%%%%%%%%%%%%%%%%%
\section{Introduction}
\label{sec:intro}

In \cite{arnowitt-deser-misner}, Arnowitt, Deser and Misner identified a key quantity in the General Relativity theory which stands for the total mass of an isolated gravitational system. Mathematically, this system is represented by a Riemannian manifold asymptotically modelled on the Euclidean space $\mathbb R^n$. That mass quantity, named the {\it{ADM mass}} after the authors, was later proved by Bartnik in \cite{bartnik} to be independent of the chosen asymptotic coordinates.

The non-negativity of the ADM mass in case the physically isolated system is time-symmetric when embedded in a Lorentzian space-time is known as the {\it{Riemannian positive mass theorem}} and was proved by Schoen and Yau in \cite{ schoen2, schoen-yau} in dimensions $3\leq n\leq 7$, and by Witten in \cite{witten} in any dimensions $n\geq 3$ assuming the manifold is spin.     
The results in \cite{schoen2, schoen-yau} played a crucial role in the proof of the remaining cases of the Yamabe problem by Schoen in \cite{schoen1}.
Extensions to high dimensions of the results in \cite{schoen2, schoen-yau} were obtained by Schoen and Yau in \cite{schoen-yau3}.  

The formulation of the Yamabe problem on manifolds with boundary by Escobar in \cite{escobar2, escobar3} motivated a version of the ADM mass and a corresponding Riemannian positive mass theorem for manifolds asymptotically modelled on the Euclidean half-space 
$$\mathbb R^n_+=\{x=(x_1,...,x_n)\in\mathbb R^n\: ;\:x_1\geq 0\}$$ 
in \cite{almaraz-barbosa-lima}. Our goal in this paper is to extend the main result in \cite{almaraz-barbosa-lima} to the case where the Riemannian metric is continuous but not smooth across a compact hypersurface. This is inspired by a similar result for the ADM mass established by Miao in \cite{miao} (see also \cite{lee-lefloch}). 
As an application, we prove a positive mass theorem for manifolds asymptotically modelled on the Euclidean quarter-space 
$$\mathbb R^n_{\mathbb L}=\{x\in\mathbb R^n\: ;\: x_1\geq 0,\, x_n\geq 0\}$$ 
where a definition of a mass quantity is introduced in \cite{almaraz-lima-mckeown}.

Let $M$ be an oriented smooth differentiable manifold with boundary $\d M$, with dimension $n\geq 3$ (where the classical positive mass theorem holds, by \cite{schoen2, schoen-yau, schoen-yau3}). Let $\Sigma\subset M$ be a compact hypersurface with boundary $\partial \Sigma = \Sigma\cap \partial M$ such that 

\begin{itemize}
\item $\Sigma$ is transversal to $\partial M$;
\item $\Sigma$ divides $M$ into two open subsets $\Omega$ and $M\backslash \overline \Omega$ of $M$ such that the closure $\overline \Omega$ is compact.
\end{itemize}
In particular, $\partial\overline\Omega=\Sigma\cup (\partial M\cap \Omega)$.
Let $g$ be a continuous Riemannian metric on $M$ satisfying
\begin{itemize}
\item $g_-:=g|_{\overline\Omega}$ is $C^{1,\alpha}(\Omega)\cap C^{2,\alpha}(\text{int}\,\Omega)$;
\item $g_+:=g|_{M\backslash\Omega}$ is $C^{1,\alpha}(M\backslash \overline\Omega)\cap C^{2,\alpha}(\text{int}\,(M\backslash \overline\Omega))$;
\item $g_-$ and $g_+$ are $C^2$ up to $\Sigma\backslash \partial\Sigma$,
\end{itemize}
for some $0<\alpha<1$. Moreover, let $\eta_g$ be the outward pointing unit normal vector to $\partial M$ and consider the mean curvatures $H_{g_-}^{\partial M}=\text{div}_{g_-}\eta_g$ and $H_{g_+}^{\partial M}=\text{div}_{g_+}\eta_g$. Let $R_{g_-}$ and $R_{g_+}$ stand for the scalar curvatures of $g_-$ and $g_+$ respectively.

Assume that $(M,g)$ is asymptotically flat with non-compact boundary in the sense of \cite{almaraz-barbosa-lima}. Namely, there exists a compact set $K\subset M$, which we will assume to contain $\Omega$, such that $M\backslash K$ is diffeomorphic to $\mathbb R^n_+$ minus a half-ball and, in the induced coordinates, $g$ satisfies 
\begin{equation}\label{decay:g}
\sum_{a,b,c,d=1}^{n}\left( 
|g_{ab}(x)-\delta_{ab}|+|x||g_{ab,c}(x)|+|x|^2|g_{ab,cd}(x)|
\right)
\leq C|x|^{-\tau},
\end{equation}
for some $C>0$ and $\tau>(n-2)/2$. Here, the comas stand for partial derivatives and 
$$
\delta_{ab}=
\begin{cases}
1 &a=b,
\\
0 &a\neq b.
\end{cases}
$$
Moreover, $R_{g_+}$ and $H_{g_+}^{\partial M}$ are integrable on $M\backslash\overline\Omega$ and $\partial M\backslash\overline\Omega$ respectively. 
We define the mass of $(M,g)$ using the same expression in  \cite{almaraz-barbosa-lima}, namely,
$$
m_{\partial M}(g)=\lim_{\rho\to\infty}
\left(
\sum_{a,b=1}^n\int_{|x|=\rho,  x_1\geq 0} \left((g_+)_{ab,b}-(g_+)_{bb,a}\right)\frac{x_a}{|x|}
+\sum_{i=1}^{n-1}\int_{|x|=\rho, \,x_1= 0} (g_+)_{in}\frac{x_i}{|x|}
\right)
$$
where we are omitting the area elements induced by the Euclidean metric. Our first main result is the following:

\begin{theorem}\label{main:thm}
Let $(M,g)$ be as above\footnote{Apart from the non-smoothness of the metric $g$ across the compact hypersurface $\Sigma$, the hypothesis on the regularity of the metric $g$ away from $\Sigma$ is weaker in the statement of Theorem \ref{main:thm} when compared to Theorem 1.3 in \cite{almaraz-barbosa-lima}. The reason is that the assumption $g\in C^\infty(M)$ made in \cite{almaraz-barbosa-lima} is unnecessary as the proofs in that paper only require $g\in C^{1,\alpha}(M)\cap C^{2,\alpha}(M\backslash \partial M)$.} and let $\nu_g$ be the unit normal vector to $\Sigma$ pointing towards the unbounded region $M\backslash\overline\Omega$ ((see Figure \ref{M})).  
Suppose that $R_{g_-}$, $R_{g_+}$, $H_{g_-}^{\partial M}$ and $H_{g_+}^{\partial M}$ are non-negative. 
Assume also that
$$
\Sigma\perp\partial M
\qquad\text{and}\qquad 
H_{g_-}^{\Sigma}\geq H_{g_+}^{\Sigma},
$$
where $H_{g_-}^{\Sigma}=\text{div}_{g_-}\nu_g$ and $H_{g_+}^{\Sigma}=\text{div}_{g_+}\nu_g$ are the mean curvatures on $\Sigma$ with respect to the regions $\overline\Omega$ and $M\backslash\Omega$ respectively. 
Then
$$m_{\partial M}(g)\geq 0.$$
If we further assume that $H_{g_-}^{\Sigma} > H_{g_+}^{\Sigma}$ at some point on $\Sigma$, then $m_{\partial M}(g)>0$.
\end{theorem}

\begin{figure}
	\includegraphics[width=\linewidth]{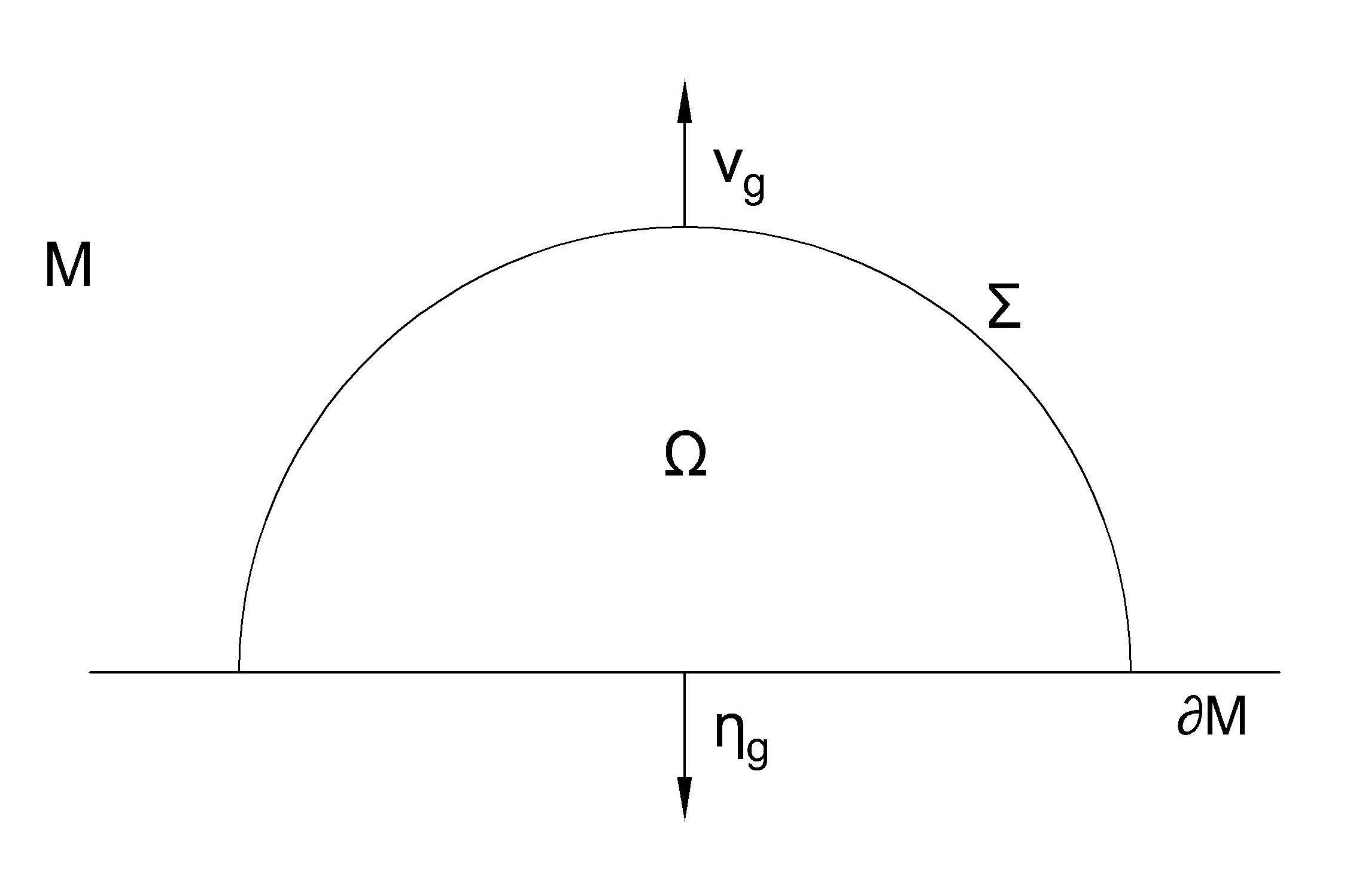}
	\caption{The region $\Omega$.}
	\label{M}
\end{figure}

Theorem \ref{main:thm} extends the positive mass theorem in \cite{almaraz-barbosa-lima}, where the case of smooth metrics is handled. As mentioned before, it is highly inspired by the results of Miao in \cite{miao} which studies the positivity of the classical ADM mass on manifolds without boundary.

The non-negativity of $R_{g_-}$, $R_{g_+}$, $H_{g_-}^{\partial M}$ and $H_{g_+}^{\partial M}$ can be viewed as dominant energy conditions for the system determined by the asymptotically flat manifold $(M,g)$. While the non-negativity of the scalar curvature is the classical dominant energy condition in General Relativity, the condition on the boundary mean curvature is the right one to be added in the presence of a non-empty boundary as observed in \cite{almaraz-lima-mari}.

An heuristic argument given in \cite[Section 2]{miao} suggests that a jump from $H^\Sigma_{g_-}$ to $H^\Sigma_{g_+}$ across $\Sigma$ (see  equation \eqref{eq:jump}) makes the scalar curvature $R_g$ to behave as a dirac delta function with support in $\Sigma$. Hence, the condition $H_{g_-}^{\Sigma}\geq H_{g_+}^{\Sigma}$ is, in some sense, related to the non-negativity of $R_g$.

The hypothesis $\Sigma\perp \partial M$ is required mainly in two parts of the manuscript. 
The first one is to mollify the metric $g$, which is only continuous across $\Sigma$, to obtain a smooth family of metrics in a small neighbourhood. 
That mollification, carried out in Section 2, is done along the integral curves of $\xi$, a vector-field near $\Sigma$ obtained in Proposition \ref{propo:vectorfield}, which is at the same time orthogonal to $\Sigma$ and tangent to $\partial M$.
%Although we believe the condition  $\Sigma\perp \partial M$ could be weakened in Theorem \ref{main:thm}, it is vital for its proof.

The second part where  $\Sigma\perp \partial M$  is used is in the application of Theorem \ref{main:thm} to the proof of a positive mass theorem on manifolds asymptotically modelled on $\mathbb R^n_{\mathbb L}$ following the mass definition in \cite{almaraz-lima-mckeown}. This is our second main result which is the content of Theorem \ref{appl:thm} below and is done by means of a doubling argument which turns the manifold, originally modelled on $\mathbb R^n_{\mathbb L}$, into one modelled on $\mathbb R^n_+$, allowing us to use Theorem \ref{main:thm}. 

\begin{theorem}\label{appl:thm}
Let $(M,\Sigma_1,\Sigma_2, g)$ be an asymptotically flat cornered space as defined in Section \ref{sec:application}.
Let $\eta_1$ and $\eta_2$ be the outward pointing unit normal vectors to $\Sigma_1$ and $\Sigma_2$ respectively and consider the mean curvatures 
$$
H_g^{\Sigma_1}=\text{div}_g\eta_1\qquad\text{and}\qquad H_g^{\Sigma_2}=\text{div}_g\eta_2.
$$
If
$$
\Sigma_1\perp\Sigma_2,
\qquad
R_g\geq 0,
\qquad
H_g^{\Sigma_1}\geq 0,
\qquad\text{and}\qquad
H_g^{\Sigma_2}\geq 0,
$$ 
then the mass $m_{\Sigma_1,\Sigma_2}(g)$, defined by formula \eqref{mass:corner}, is non-negative.
\end{theorem}

The proof of Theorem \ref{main:thm} involves an approximation argument and conformal changes of metric  in order to apply the positive mass theorem in \cite{almaraz-barbosa-lima}. We first approximate the only continuous metric by some smooth metrics in a tubular neighbourhood of $\Sigma$. Direct calculations show that the scalar curvatures are bounded apart from a jump of the mean curvature on $\Sigma$, and the mean curvatures on $\partial M$ are bounded as well. We then use a conformal transformation of metric to restore non-negativity of those curvatures, which allows us to apply the positive mass theorem in \cite{almaraz-barbosa-lima} for case of smooth metrics on manifolds with non-compact boundary. Finally the non-negativity of the mass comes from the convergence of the metrics. 

This paper is organized as follows. In Section \ref{sec:smooth}, we mollify our original metric to obtain the approximating family of smooth metrics, mentioned in the previous paragraph, and derive some estimates of their scalar and boundary mean curvatures. These estimates are then used in Section \ref{sec:pf:main:thm} to obtain a family of smooth metrics with non-negative scalar and boundary mean curvatures via conformal transformations. We then present the proof of Theorem \ref{main:thm} in the same section.  In Section \ref{sec:application} of the paper, we present the proof of Theorem \ref{appl:thm}. Finally in the appendix, we collect some technical results used in Section \ref{sec:application}.

%%%%%%%%%%%%%%%%%%%%%%%%%%%%%%%%%%%%%%%%%%%%

\section{Smoothing the metric across $\Sigma$}
\label{sec:smooth}
Let $(M,g)$ be as in Theorem \ref{main:thm}.
In this section, we construct a smooth family of metrics that approximate the only continuous metric $g$. In order to restore the non-negativity of scalar curvature and boundary mean curvature via a conformal change of metric in the next section, we prove some control on those curvatures in Proposition \ref{propo:control:curv} below which is the main result of this section.

In order to mollify the metric $g$ on a neighbourhood of $\Sigma$, we need the following result, which uses the orthogonality $\Sigma\perp\partial M$:

\begin{proposition}\label{propo:vectorfield}
There is a vector field $\xi$ on an open neighbourhood $U\subset M$ of $\Sigma$, with the same regularity as the metric $g$, such that 
\begin{itemize}
\item $\xi$ is tangent to $\partial M$;
\item $\xi=\nu_g$ on $\Sigma$; 
%\item $\xi\in C^{1,\alpha}(U)\cap C^{2,\alpha}(U\backslash\partial M)$,
\end{itemize}
for some $0<\alpha<1$.
\end{proposition}

The proof of Proposition \ref{propo:vectorfield} relies on the following local result:

\begin{lemma}\label{lemma:local:coord}
Let $p\in\partial\Sigma$. There exist an open neighbourhood $U_p\subset M$ of $p$ and a vector field $\xi_p$ on that neighbourhood, with the same regularity as the metric $g$, such that
\begin{itemize}
\item $\xi_p$ is tangent to $\partial M$;
\item $\xi_p=\nu_g$ on $\Sigma\cap U_p$; 
%\item $\xi_p\in C^{1,\alpha}(U_p)\cap C^{2,\alpha}(U_p\backslash\partial M)$,
\end{itemize}
for some $0<\alpha<1$.
\end{lemma}

\begin{proof}
Choose a neighbourhood $U_p\subset M$ of $p$ and coordinates $\Phi=(x_1,...,x_n)$,
$$
\Phi: U_p\to\mathbb R^n_{+}
$$
satisfying
$$
\Phi(\Sigma\cap U_p)\subset\{x_n=0\},\qquad
\Phi(\partial M\cap U_p)\subset\{x_1=0\},\qquad
\Phi(\partial \Sigma\cap U_p)\subset\{x_1=x_n=0\}.
$$
Without loss of generality, assume that $\partial_{x_n}$ points towards the same direction as $\nu_g$.

If $x=(x_1,...,x_n)\in \mathbb R^n$, set $\bar x=(x_1,...,x_{n-1},0)$. For $x\in\Phi(U_p)$ define
$$
X(x)=\partial_{x_n}-\sum_{i=1}^{n-1}(\Phi_*g)_{\bar x}(\partial_{x_i}, \partial_{x_i})^{-1}(\Phi_*g)_{\bar x}(\partial_{x_n}, \partial_{x_i})\partial_{x_i}.
$$
Observe that $X$ is tangent to $\{x_1=0\}$. Indeed, $(\Phi_*g)_{\bar x}(\partial_{x_n}, \partial_{x_1})=0$ for $x_1=0$, as $\Sigma\perp \partial M$.
The lemma is proved by setting 
$$
\xi_p=\frac{1}{\|\Phi^*X\|_{g}}\Phi^*X\,.
$$
\end{proof}

\begin{proof}[Proof of Proposition \ref{propo:vectorfield}]
As $\partial\Sigma$ is compact we can choose finitely many points $p_1,...,p_k\in\partial\Sigma$ and open neighbourhoods $U_{1},...,U_{k}$ of these points in $M$, having vector fields $\xi_{1}$, ..., $\xi_{k}$, respectively, given by Lemma \ref{lemma:local:coord}.
Choose an open subset $U_0$ of $M$, disjoint of $\partial M$ and containing $\Sigma\backslash\left(\cup_{A=1}^{k}U_{A}\right)$, and a vector field $\xi_0$ on $U_0$ such that $\xi_0=\nu_g$ on $\Sigma\cap U_0$.

Now, $\{U_A\}_{A=0}^{k}$ is an open covering of a compact neighbourhood of $\Sigma$ for which we choose a partition of unity $\{\chi_A\}_{A=0}^k$. Set 
$$
\xi=\sum_{A=0}^{k}\chi_A\xi_A
$$
which clearly satisfies the required properties.
\end{proof}

Using the notation of Proposition \ref{propo:vectorfield}, let $\Phi_t$ be the integral flow of $\xi$, which is defined on $U$. Choosing $U$ smaller if necessary, we can assume that the map
$$
\Phi:\Sigma\times I\to U
$$
$$
(x,t)\mapsto \Phi_t(x)
$$
is a homomorphism, where $I=(-2\epsilon,2\epsilon)$ and $\epsilon>0$ is small. 
We identify $U$ with $\Sigma\times I$, referring to points on $U$ as $(x,t)\in \Sigma\times I$, and use in $U$ the differential structure provided by that identification, possibily inducing a new differential structure on $M$ still denoted by $M$. For each $t\in I$, set 
$$
\Sigma_t=\left\{(x,t)\: ;\:x\in\Sigma\}=\Sigma\times \{t\right\}.
$$

Similar to \cite{miao}, we proceed to construct a family of $C^2$ metrics $g_{\delta}$, for $\delta>0$ small, which gets uniformly close to $g$ as $\delta\to 0$ and coincides with $g$ outside $\Sigma\times \left(-\frac{\delta}{2},\frac{\delta}{2}\right)$.
We choose $\chi(t) \in C_c^{\infty}([-1,1])$ to be a standard mollifier on $\mathbb{R}$ such that
\begin{equation*}
	0 \leq \chi \leq 1 \quad \text{and} \quad \int_{-1}^{1} \chi(t) dt = 1.
\end{equation*}
Let $\sigma(t) \in C_c^{\infty}([- \frac{1}{2}, \frac{1}{2}])$ be another cut-off function such that
\begin{equation*}
	\sigma(t) = 
	\begin{cases} 
		%0 \leq \sigma(t) \leq \frac{1}{100} & t \in \mathbb{R}^1 \\ 
		\sigma(t) = \frac{1}{100}, & |t| < \frac{1}{4}, \\ 
		0 < \sigma(t) \leq \frac{1}{100}, & \frac{1}{4} < |t| < \frac{1}{2}.
	\end{cases}
\end{equation*}
Given any $0 < \delta \ll \epsilon$, let
\begin{equation}
	\sigma_{\delta}(t) = \delta^2 \sigma \left( \frac{t}{\delta} \right)
\end{equation}
and define
\begin{align}\label{def:g:delta}
	g_{\delta}(x,s) &= \int g(x,s - \sigma_{\delta}(s)t)\, \chi(t)\, dt, \quad s \in (-\epsilon, \epsilon),
	\\&= 
	\begin{cases} 
		\int g(x,t) \displaystyle\frac{1}{\sigma_{\delta}(s)} \chi \left(\frac{s - t}{\sigma_{\delta}(s)} \right)  dt, & \sigma_{\delta}(s) > 0. \\
		g(x,s), & \sigma_{\delta}(s) = 0,
	\end{cases}\notag
\end{align}
for $(x,s)\in \Sigma\times (-\epsilon, \epsilon)$ and $g_{\delta}=g$ when $(x,s)\notin \Sigma\times (-\epsilon, \epsilon).$ 
Observe that $g_\delta$ coincides with $g$ outside $\Sigma\times \left(-\frac{\delta}{2},\frac{\delta}{2}\right)$ and also that $\nabla^k g_{\delta}$ converges to $\nabla^k g$ in $L^\infty\left(\Sigma\times \left(-\frac{\delta}{2},\frac{\delta}{2}\right)\right)$ as $\delta\to 0$ for $k=0,1,2,...\,$.

\begin{remark}
Let us now point out one of the key points of our argument. In \cite{miao}, the identification of the tubular neighbourhood $U$ with $\Sigma\times I$ is isometric as it is done by means of geodesics orthogonal to $\Sigma$. 
%Since the metric is non-smooth across $\Sigma$, this identification changes the underlying differentiable structure of the manifold. 
In our case, due to the presence of the boundary $\partial M$, the neighbourhood $U$ is no longer a geometric product. Instead of geodesics, we use the integral curves of $\xi$ which are orthogonal to $\Sigma$ only at $t=0$. 
%Since $\xi$ is smooth, the identification does not change the underlying differentiable structure. 
As $U$ is not a geometric product, we need to control some additional terms in the scalar curvature when compared to \cite{miao}. We also need to control the mean curvature of $\partial M$. 
\end{remark}
The control on the scalar curvature and the boundary mean curvature is the content of the next proposition.

\begin{proposition}\label{propo:control:curv}
Denote by $R_{g_{\delta}}$ the scalar curvature of $M$ and by $H^{\d M}_{g_{\delta}}$ the mean curvature of $\partial M$ with respect to the metrics $g_\delta$ constructed above. Then
\begin{align*}
	R_{g_{\delta}}(x,t) &= O(1), \quad \text{for } (x,t) \in \Sigma \times \left\{\frac{\delta^2}{100} < |t| \leq \frac{\delta}{2} \right\},  \\
	R_{g_{\delta}}(x,t) &= O(1) + \left( H_{g_-}^\Sigma(x) - H_{g_+}^\Sigma(x) \right) \frac{100}{\delta^2} \chi \left( \frac{100 t}{\delta^2} \right), \\
	& \quad \text{for } (x,t) \in \Sigma \times \left[ -\frac{\delta^2}{100}, \frac{\delta^2}{100} \right], \\
	H^{\d M}_{g_{\delta}}(x,t) &= O(1), \quad \text{for } (x,t) \in \Sigma \times \left\{0\leq |t| \leq \frac{\delta}{2} \right\},  
\end{align*}
where $O(1)$ represents quantities that are bounded by constants depending on $g$, but not on $\delta$.
\end{proposition}

\begin{proof}
	By the Gauss equation and the first variation formula of $H^{\Sigma_t}_{g_{\delta}}$, the scalar curvature satisfies
	\begin{equation}\label{eq:jump}
	R_{g_{\delta}}(x,t)=R_{g_{\delta}}^{\Sigma_t}(x,t)-(|A^{\Sigma_t}_{{g_{\delta}}}(x,t)|^2+H^{\Sigma_t}_{g_{\delta}}(x,t)^2)-2\frac{\d}{\d \nu_t}H^{\Sigma_t}_{{g_{\delta}}}(x,t),
	\end{equation}
	where $\nu_t$ is the unit normal vector along $\Sigma_t$ pointing towards the unbounded region, $R_{g_{\delta}}^{\Sigma_t}$ is the scalar curvature and $A^{\Sigma_t}_{g_{\delta}}$ is the second fundamental form of $\Sigma_t$, all with respect to $g_\delta$. 

First of all, $R_{g_{\delta}}^{\Sigma_t}$ is bounded by constants depending only on $g$ since $g_{\delta}$ is uniformly close to $g$. Now for $A^{\Sigma_t}_{{g_{\delta}}}$, we choose coordinates such that $\partial_n=\partial_t$ and calculate it as
	$$(A^{\Sigma_t}_{g_{\delta}})_{ij}=-g_\delta(\nabla_{\d_j}\d_i, \nu_t)=-g_\delta\big(\sum_{a,b=1}^{n}\Gamma_{ij}^a\d_a, \nu^b_t\d_b\big),$$
	where $i, j=1,...,n-1$, $\Gamma_{ij}^a$ are the Christoffel symbols of the metric $g_\delta$, and 
$$
\nu_t^a=\sum_{a=1}^{n}\frac{(g_{\delta})^{na}}{\sqrt{(g_{\delta})^{nn}}}.
$$ 
After a direct computation,
\begin{align}\label{eq:Adelta}	
(A^{\Sigma_t}_{g_{\delta}})_{ij}=
&\frac{1}{2}\(\sum_{k=1}^{n-1}\frac{(g_{\delta})^{nk}}{\sqrt{(g_{\delta})^{nn}}}(g_{\delta})_{ij,k}+\sqrt{(g_{\delta})^{nn}}(g_{\delta})_{ij,n}\)
\\
&-\frac{1}{2}\(\sum_{k=1}^{n-1}\frac{(g_{\delta})^{nk}}{\sqrt{(g_{\delta})^{nn}}}((g_{\delta})_{ik,j}+(g_{\delta})_{jk,i})+\sqrt{(g_{\delta})^{nn}}((g_{\delta})_{in,j}+(g_{\delta})_{jn,i})\).\notag
\end{align}
	It follows from the definition \eqref{def:g:delta} that
	$$(g_{\delta})_{ij,k}(x,t)=\frac{\d}{\d x_k}\int g_{ij}(x,t-\sigma_{\delta}(t)s)\chi(s)ds=\int g_{ij,k}(x,t-\sigma_{\delta}(t)s)\chi(s)ds$$
	is bounded by constants depending only on $g$. Same holds for the terms $(g_{\delta})_{in,j}(x,t).$ As for $$(g_{\delta})_{ij,n}(x,t)=\frac{\d}{\d t}\int g_{ij}(x,t-\sigma_{\delta}(t)s)\chi(s)ds,$$ one calculates as in \cite[formulas (20)-(22)]{miao} to obtain for $t\in (-\epsilon, \epsilon)$,
	$$(g_{\delta})_{ij,n}(x,t)=\int g_{ij,n}(x,t-\sigma_{\delta}(t)s)\(1-s\delta\sigma'\(\frac{t}{\delta}\)\)\chi(s)ds$$
	which again is bounded by constants depending on $g$. This gives a control on $A^{\Sigma_t}_{g_{\delta}}$. A similar control holds for $H^{\Sigma_t}_{g_{\delta}}$ since $H^{\Sigma_t}_{g_{\delta}}=\sum_{i,j=1}^{n-1}(g_{\delta})^{ij}(A^{\Sigma_t}_{g_{\delta}})_{ij}$.
		Now let us look at the derivative
	\begin{equation}\label{eq:dH}
		\frac{\d}{\d \nu_t}H^{\Sigma_t}_{g_{\delta}}=\sum_{i,j=1}^{n-1}\frac{\d}{\d \nu_t}(g_{\delta})^{ij}(A^{\Sigma_t}_{g_{\delta}})_{ij}+\sum_{i,j=1}^{n-1}(g_{\delta})^{ij}\frac{\d}{\d \nu_t}(A^{\Sigma_t}_{g_{\delta}})_{ij}.
	\end{equation}
	It follows from equation \eqref{eq:Adelta} that all terms of \eqref{eq:dH} are bounded by constants depending only on $g$ except for the term 
	$$(g_{\delta})^{nn}(g_{\delta})_{ij,nn}.$$
	One can calculate similarly as in \cite[formulas (24)-(25), (27)-(30)]{miao}  to obtain 
	$$\sum_{i,j=1}^{n-1}(g_{\delta})^{nn}(g_{\delta})^{ij}(g_{\delta})_{ij,nn}(x,t)=O(1)+\left( H_{g_+}^\Sigma(x) - H_{g_-}^\Sigma(x) \right) \frac{100}{\delta^2}\chi\(\frac{100t}{\delta^2}\),$$ inside of $\Sigma\times[-\frac{\delta^2}{100}, \frac{\delta^2}{100}]$. 
	The above calculations give us the estimate for $R_{g_{\delta}}$. Now let us turn to 
	$$H^{\d M}_{g_{\delta}}=\sum_{\alpha,\beta=2}^{n}(g_{\delta})^{\alpha\beta}(A^{\partial M}_{g_{\delta}})_{\alpha\beta},$$
	where 
$$
(A^{\partial M}_{g_{\delta}})_{\alpha\beta}=-g_\delta(\nabla_{\d_{\alpha}}\d_{\beta},  \eta_{\d M}) \qquad\text{and}\qquad\eta_{\d M}=-\sum_{a=1}^{n}\frac{(g_{\delta})^{1a}}{\sqrt{(g_{\delta})^{11}}}\d_a.
$$
	Similarly as for the calculations of $H^{\Sigma_t}_{g_{\delta}}$, we can conclude that $H^{\d M}_{g_{\delta}}$ is also bounded by constants depending only on $g$. This finishes the proof of Proposition \ref{propo:control:curv}.
\end{proof}

%%%%%%%%%%%%%%%%%%%%%%%%%%%%%%%%%%%%%%%%%%%%%%%%%%%%%%%%%%%%%%%%%%%

%%%%%%%%%%%%%%%%%%%%%%%%%%%%%%%%%%%%%%%%%%%%%%%%%%%%%%%%%%%%%%%%%%%%%%%%%%%%%%%%%%%%%%

\section{Proof of Theorem \ref{main:thm}}
\label{sec:pf:main:thm}

The proof of the first statement of Theorem \ref{main:thm}, namely that $m_{\partial M}(g)\geq 0$, consists of two steps. The first one is the solution of a linear equation aiming to find a family of smooth functions $u_\delta>0$ such that $\widetilde g_\delta=u_\delta^{\frac{4}{n-2}}g_\delta$ has non-negative scalar curvature on $M$ and non-negative mean curvature on $\partial M$. This is the content of Proposition \ref{propo:lim:u} which ensures that $m_{\partial M}(\widetilde g_\delta)\geq 0$ by \cite{almaraz-barbosa-lima}. The second step is provided by Proposition \ref{propo:conv:mass} which says that the mass of $\widetilde g_\delta$ converges to the mass of $g$ as $\delta\to 0$, thus proving the non-negativity of the latter. Finally, in Subsection \ref{subsec:positive} we prove the last statement of Theorem \ref{main:thm}.

We first fix some notations here. For any given function $f$, we define the positive and negative parts of it as $f^+$ and $f^-$. In this sense, we have $f=f^+-f^-$ and $|f|=f^++f^-$. 

For an integer $k\geq 0$, and any numbers $q\geq 1$ and $\gamma\in\mathbb R$, we define the weighted Sobolev space $L^{q}_{k,\gamma}(M)$ to be the set of functions $w$ for which the norm
\begin{equation*}
\|w\|_{L^q_{k,\gamma}(M)}=
\begin{cases}
\sum_{i=0}^{k}\left(\int_M (|r^{i-\gamma}\nabla_g^i w|)^q r^{-n}dv_g\right)^{\frac{1}{q}}, &q<\infty, \\
\sum_{i=0}^{k}\mathrm{ess\,sup}_M(r^{i-\gamma}|\nabla_g^i w|), & q=\infty,
\end{cases}
\end{equation*}
is finite. Here, $r(x)$ is any smooth positive extension of the asymptotic parameter $|x|$ to $M$. 
We also define the {\em weighted} $C^k$ {\em space} $C^{k}_\gamma(M)$ as the set of $C^k$ functions $w$ on $M$ for which the norm  
$$
\|w\|_{C^k_\gamma(M)}=\sum_{i=0}^k\sup_M r^{-\gamma+i}|\nabla^i w|
$$
is finite. Moreover, if $0<\alpha<1$, we define the {\em weighted Hölder space} $C^{k,\alpha}_{\gamma}$ as the set of functions $w\in C^{k}_{\gamma}(M)$ such that   
the norm 
$$
\|w\|_{C^{k,\alpha}_\gamma(M)}=\|w\|_{C^k_\gamma(M)}+
\sup_{x,y}\left\{\left(\min \{r(x),r(y)\}\right)^{-\gamma+k+\alpha}\frac{|\nabla^k w(x)-\nabla^k w(y)|}{|x-y|^\alpha}\right\}
$$
is finite. Here, the supreme is over all $x\neq y$ in $M$ such that $y$ is contained in a normal coordinate neighbourhood of $x$, and $\nabla^k w(y)$ is the tensor at $x$ obtained by the parallel transport along the radial geodesic from $x$ to $y$.

\subsection{Conformal deformation}

Set $c_n=(n-2)/(4(n-1))$ and also $H_{g_{\delta}}=H^{\d M}_{g_{\delta}}$ in this subsection for simplicity. We wish to find solutions $u_\delta>0$ to the system
\begin{align}\label{eq:u:delta}
	\begin{cases}
		\Delta_{g_{\delta}} u_{\delta}+c_nR^-_{g_{\delta}}u_{\delta}=0,&M,\\
		\displaystyle \frac{\d  u_{\delta}}{\d \eta_{g_\delta}}-2c_nH^-_{g_{\delta}}u_{\delta}=0, &\partial M,
	\end{cases}
\end{align}
of the form $u_\delta=1+w_\delta$ such that $w_\delta\to 0$ uniformly as $\delta\to 0$.

\begin{proposition}\label{propo:lim:u}
Let $2-n<\gamma<0$ and $q>1$. For small $\delta>0$, there exists $w_\delta\in L^{q}_{2,\gamma}(M)$ satisfying
\begin{align}\label{eq:w delta}
	\begin{cases}
		\Delta_{g_{\delta}} w_{\delta}+c_nR^-_{g_{\delta}}w_{\delta}=-c_nR^-_{g_{\delta}}, &M,\\
		\displaystyle\frac{\d  w_{\delta}}{\d \eta_{g_\delta}}-2c_nH^-_{g_{\delta}}w_{\delta}=2c_nH^-_{g_{\delta}}, &\partial M,
	\end{cases}
\end{align}
with $\lim\limits_{\delta\rightarrow 0}\|w_{\delta}\|_{L^{\infty}(M)}=0$.
\end{proposition}
\begin{proof}
Observe that the coefficients of $g$ are in the Sobolev space $W_{loc}^{1,p}(M)$ for any $p>1$ and so we can consider the operator 
$$
T:L^q_{2,\gamma}(M)\to L^q_{0,\gamma-2}(M) \times L^{q}_{1-1/q,\gamma-1} (\partial M)
$$
given by $T w=\left( \Delta_{g}w, \d w/\d\eta_{g}\right)$.
Using \cite[Proposition 2.2]{bartnik} and the doubling argument in \cite[Proposition 3.4]{almaraz-barbosa-lima} we see that $T$ is an isomorphism.
Consider also the operators
$$
T_\delta w=\left(\Delta_{g_\delta}w+c_nR^-_{g_{\delta}} w, \frac{\partial w}{\partial \eta_{g_\delta}}-2c_nH_{g_{\delta}}^-w \right).
$$

One calculates using Hölder inequality for weighted Sobolev spaces (see for example Theorem 1.2 in \cite{bartnik}) that for some constant $C>0$: 
\begin{align*}
	\|(\Delta_{g_{\delta}}-\Delta_{g})w\|_{L^q_{0,\gamma-2}(M)}\leq  C\|g_{\delta}-g\|_{L^{\infty}_{1,0}(M)}\|w\|_{L^q_{2,\gamma}(M)}\:,
\end{align*}
\begin{align*}
	\|R_{g_{\delta}}^-w\|_{L^q_{0,\gamma-2}(M)}\leq  C\|g_{\delta}-g\|_{L^{\infty}_{2,0}(M)}\|w\|_{L^q_{2,\gamma}(M)}\:,
\end{align*}
\begin{align*}
	\big\|\frac{\partial w}{\partial \eta_{g_\delta}}-\frac{\partial w}{\partial \eta_{g}}\big\|_{L^{q}_{1-1/q,\gamma-1} (\partial M)}\leq  C\|g_{\delta}-g\|_{L^{\infty}_{1,0}(M)}\| w\|_{L^q_{2,\gamma}(M)}\:,
\end{align*}
\begin{align*}
	\|H_{g_{\delta}}^-w\|_{L^{q}_{1-1/q,\gamma-1} (\partial M)}\leq 
	C\|g_{\delta}-g\|_{L^{\infty}_{2,0}(M)}\| w\|_{L^q_{2,\gamma}(M)}\:.
\end{align*}
Since $\|g_{\delta}-g\|_{L^{\infty}_{2,0}(M)}\to 0$, as $\delta\to 0$, $T_\delta$ converges to $T$ in the space of continuous operators from $L^q_{2,\gamma}(M)$ to $L^q_{0,\gamma-2}(M) \times L^{q}_{1-1/q,\gamma-1} (\partial M)$ so that  $T_\delta$ is an isomorphism for small $\delta$. This proves the first statement of Proposition \ref{propo:lim:u}.

Let us now prove that $\lim\limits_{\delta\rightarrow 0}\|w_{\delta}\|_{L^{\infty}(M)}=0$.
We multiply the equations \eqref{eq:w delta} by $w_{\delta}$ and integrate to obtain
\begin{align*}
\int_{M}(w_{\delta}\Delta_{g_{\delta}}w_{\delta}+c_nR_{g_{\delta}}^-w_{\delta}^2)&-\int_{\d M} \left(w_{\delta}\frac{\d w_{\delta}}{\d \eta_{g_\delta}}-2c_nH^-_{g_{\delta}}w_{\delta}^2\right)\\
=&-\int_{M}c_nR^-_{g_{\delta}}w_{\delta}-\int_{\d M}2c_nH_{g_{\delta}}^-w_{\delta}.
\end{align*}
Here and in the rest of this proof, we are omitting the volume and area elements with respect to $g_\delta$.
Using integration by parts and Hölder inequality we have
\begin{align}\label{ineq:w}
\int_{M}|\nabla_{g_{\delta}}w_{\delta}|^2\leq c_n&\(\int_{M}(R_{g_{\delta}}^-)^{\frac{n}{2}}\)^{\frac{2}{n}}\(\int_{M}|w_{\delta}|^{\frac{2n}{n-2}}\)^{\frac{n-2}{n}}+c_n\(\int_{M}(R_{g_{\delta}}^-)^{\frac{2n}{n+2}}\)^{\frac{n+2}{2n}}\(\int_{M}|w_{\delta}|^{\frac{2n}{n-2}}\)^{\frac{n-2}{2n}}\notag
\\
&+2c_n\(\int_{\d M}(H^-_{g_{\delta}})^{n-1}\)^{\frac{1}{n-1}}\(\int_{\d M}|w_{\delta}|^{\frac{2(n-1)}{n-2}}\)^{\frac{n-2}{n-1}}\notag
\\
&+2c_n\(\int_{\d M}(H^-_{g_{\delta}})^{\frac{2(n-1)}{n}}\)^{\frac{n}{2(n-1)}}\(\int_{\d M}|w_{\delta}|^{\frac{2(n-1)}{n-2}}\)^{\frac{n-2}{2(n-1)}}.
\end{align}
It follows from Sobolev and Sobolev trace inequalities that
$$\(\int_{M}|w_{\delta}|^{\frac{2n}{n-2}}\)^{\frac{n-2}{n}}+\(\int_{\d M}|w_{\delta}|^{\frac{2(n-1)}{n-2}}\)^{\frac{n-2}{n-1}}\leq (C_{\delta}+C'_{\delta})\int_{M}|\nabla_{g_{\delta}}w_{\delta}|^2$$
where $C_{\delta},C'_{\delta}$ are the corresponding Sobolev constants for $(M,g_{\delta})$. Therefore,
\begin{align*}
\(\int_{M}|w_{\delta}|^{\frac{2n}{n-2}}\)^{\frac{n-2}{n}}&+\(\int_{\d M}|w_{\delta}|^{\frac{2(n-1)}{n-2}}\)^{\frac{n-2}{n-1}}
\\ 
&\leq c_n(C_{\delta}+C'_{\delta})\(\int_{M}(R^-_{g_{\delta}})^{\frac{n}{2}}\)^{\frac{2}{n}}\(\int_{M}|w_{\delta}|^{\frac{2n}{n-2}}\)^{\frac{n-2}{n}}+\frac{1}{2}c_n^2(C_{\delta}+C'_{\delta})^2\(\int_{M}(R^-_{g_{\delta}})^{\frac{2n}{n+2}}\)^{\frac{n+2}{n}}
\\
&\hspace{0.4cm}+\frac{1}{2}\(\int_{M}|w_{\delta}|^{\frac{2n}{n-2}}\)^{\frac{n-2}{n}}+2c_n(C_{\delta}+C'_{\delta})\(\int_{\d M}(H^-_{g_{\delta}})^{n-1}\)^{\frac{1}{n-1}}\(\int_{\d M}|w_{\delta}|^{\frac{2(n-1)}{n-2}}\)^{\frac{n-2}{n-1}}
\\
&\hspace{0.6cm}+2c_n^2(C_{\delta}+C'_{\delta})^2\(\int_{\d M}(H^-_{g_{\delta}})^{\frac{2(n-1)}{n}}\)^{\frac{n}{n-1}}+\frac{1}{2}\(\int_{\d M}|w_{\delta}|^{\frac{2(n-1)}{n-2}}\)^{\frac{n-2}{n-1}}.
\end{align*}
Now for $\delta$ small enough one obtains
$$\(\int_{M}|w_{\delta}|^{\frac{2n}{n-2}}\)^{\frac{n-2}{n}}+\(\int_{\d M}|w_{\delta}|^{\frac{2(n-1)}{n-2}}\)^{\frac{n-2}{n-1}}\leq C \(\int_{M}(R^-_{g_{\delta}})^{\frac{2n}{n+2}}\)^{\frac{n+2}{n}}+C\(\int_{\d M}(H^-_{g_{\delta}})^{\frac{2(n-1)}{n}}\)^{\frac{n}{n-1}}.$$
Observe that the right hand side of this inequality goes to zero as $\delta\to 0$.
It then follows from \cite[Proposition A.2]{almaraz-sun} that
$$\lim_{\delta\to 0}\sup_{M}|w_{\delta}|=0.$$ 
%Since $g_{\delta}$ coincides with $g$ outside a compact set, we have that $R^-_{g_{\delta}}\to 0$ in $C^{0,\alpha}_{-2}(M)$ and $H^-_{g_{\delta}}\to 0$ in $C^{1,\alpha}_{-1}(M)$ as $\delta\to 0$ because $g$ has non-negative curvatures. The fact that $\lim\limits_{\delta\rightarrow 0}\|w_{\delta}\|_{C^{2,\alpha}_{0}(M)}=0$ follows from the Schauder estimates
%$$\|w_{\delta}\|_{C^{2,\alpha}_{\beta}(M)}\leq C\left(\|\Delta_{g_{\delta}}w_{\delta}\|_{C^{0,\alpha}_{\beta-2}(M)}+\big\|\frac{\partial w_{\delta}}{\partial \eta_{g_\delta}}\big\|_{C^{1,\alpha}_{\beta-1}(\d M)}+\|w_{\delta}\|_{C^{0}_{\beta}(M)}\right)$$ in weighted Hölder spaces for any $\beta\in \R$ 
%(see for example Lemma A.2 in \cite{almaraz-barbosa-lima} or Theorem 9.2 in \cite{lee-parker}).
This finishes the proof. 
\end{proof}

%%%%%%%%%%%%%%%%%%%%%%%%%%%%%%%%%%%%%%%%%%%%%%%%%%%%%%%%%%%%%%%%%%%%%%%%%%%%%%%%%%%%%

\subsection{Convergence of the mass} 

It follows from Proposition \ref{propo:lim:u} and the conformal transformation laws that the problem \eqref{eq:u:delta} has a solution $u_\delta=1+w_\delta$ such that $\widetilde{g_{\delta}}=u_{\delta}^{\frac{4}{n-2}}g_{\delta}$ satisfies
\begin{equation}\label{eq:Rgdelta}
R_{\widetilde g_\delta}=c_n^{-1}u_{\delta}^{-\frac{n+2}{n-2}}(-\Delta_{g_{\delta}}u_{\delta}+c_nR_{g_{\delta}}u_{\delta})=u_{\delta}^{-\frac{4}{n-2}}R_{g_{\delta}}^+\geq 0
\end{equation}
and 
\begin{equation}\label{eq:Hgdelta}
H^{\d M}_{\widetilde g_\delta}=\frac{1}{n-1}c_n^{-1}\(\frac{\d u_{\delta}}{\d \eta_{g_\delta}}+2c_nH^{\d M}_{g_{\delta}}u_{\delta}\)=\frac{1}{n-1}u_{\delta}^{-\frac{2}{n-2}}(H^{\d M}_{g_{\delta}})^+\geq 0.
\end{equation}

Moreover, $(M, \widetilde g_\delta)$ is asymptotically flat (with non-compact boundary).
Indeed, according to \cite[Proposition 3.2]{almaraz-barbosa-lima}, the space $L^q_{k,\gamma}(M)$ is embedded in an appropriate Hölder weighted space so that $w_\delta(x)=O(|x|^{\gamma})$ and $\partial w_\delta(x)=O(|x|^{\gamma-1})$ for some $2-n<\gamma<(2-n)/2$. Then the decay \eqref{decay:g} holds. The integrability of $R_{\widetilde g_\delta}$ and $H^{\d M}_{\widetilde g_\delta}$ follows from \eqref{eq:Rgdelta}, \eqref{eq:Hgdelta} and Proposition~\ref{propo:lim:u}. By the positive mass theorem in \cite{almaraz-barbosa-lima}\footnote{Although the main result in \cite{almaraz-barbosa-lima} is stated for smooth metrics, they also hold when their coefficients are in $C^{2,\alpha}$ in the interior of the manifold and $C^{1,\alpha}$ up to the boundary. See the footnote for Theorem \ref{main:thm}.}, we have $m_{\partial M}(\widetilde g_\delta)\geq 0$. The fact that $m_{\partial M}(g)\geq 0$ follows from the next proposition.

\begin{proposition}\label{propo:conv:mass}
We have
$$
\lim_{\delta\to 0} m_{\partial M}(\widetilde g_\delta)=m_{\partial M}(g).
$$
\end{proposition}
\begin{proof}
Let us again set $H_{g_{\delta}}=H_{g_{\delta}}^{\d M}$ in the proof for simplicity. We define 
$$B^+_{\rho}=\{z\in\Rn\,;\:|z|<{\rho} \}, \qquad  S^+_{\rho}=\{z\in\Rn\,;\:|z|={\rho} \},$$
and $$D_{\rho}=B^+_{\rho}\cap \d\Rn=\{z\in\d\Rn\,;\:|z|<{\rho} \}.$$ It follows from the definition of mass that 
\begin{align*}
m_{\partial M}&(\widetilde g_\delta)-m_{\partial M}(g)=m_{\partial M}(\widetilde g_\delta)-m_{\partial M}(g_\delta)\\
&=\lim_{\rho\to\infty}\int_{S_\rho^+}
\left(
\partial_b (u_\delta^{\frac{4}{n-2}})(g_\delta)_{ab} 
-\partial_a (u_\delta^{\frac{4}{n-2}})(g_\delta)_{bb} 
\right)
\frac{x_a}{|x|} d\sigma\\
&=(1-n)\lim_{\rho\to \infty}\int_{S_\rho^+}\partial_a(u_\delta^{\frac{4}{n-2}})\frac{x_a}{|x|}d\sigma
=-\frac{4(n-1)}{n-2}\lim_{\rho\to \infty}\int_{S_\rho^+}\frac{\partial u_\delta}{\partial \mu} d\sigma, 
\end{align*}
where $\mu$ stands for the Euclidean outward pointing normal vector to $S_\rho^+$ and $d\sigma$ is the Euclidean area element.

On the other hand, we integrate by parts the equations \eqref{eq:u:delta} to obtain  
\begin{align*}
0&=\int_{B_\rho^+}\left( \frac{4(n-1)}{n-2}u_\delta(-\Delta_{g_\delta}u_\delta)-R^-_{g_\delta}u_{\delta}^2\right)dv_{g_\delta}\\
&\hspace{0.4cm}+\int_{D_\rho}\left( \frac{4(n-1)}{n-2}u_\delta \frac{\partial u_\delta}{\partial \eta_{g_\delta}}-2H_{g_\delta}^{-} u_{\delta}^2\right) d\sigma_{g_\delta}\\
&=\int_{B_\rho^+}\left( \frac{4(n-1)}{n-2}|\nabla_{g_\delta}u_\delta|_{g_\delta}^2-R^-_{g_\delta}u_{\delta}^2\right)dv_{g_\delta}\
-\int_{D_\rho}2H_{g_\delta}^{-}u_{\delta}^2 d\sigma_{g_\delta}\\
&\hspace{0.4cm}-\int_{S_\rho^+}\frac{4(n-1)}{n-2}u_\delta \frac{\partial u_\delta}{\partial \mu_{g_\delta}} d\sigma_{g_\delta}.
\end{align*}
Taking the limit as $\rho\to\infty$, we get
\begin{align*}
&\frac{4(n-1)}{n-2}\lim_{\rho\to\infty}\int_{S_\rho^+}u_\delta \frac{\partial u_\delta}{\partial \mu_{g_\delta}} d\sigma
\\
&=\int_{M}\left( \frac{4(n-1)}{n-2}|\nabla_{g_\delta}u_\delta|_{g_\delta}^2-R^-_{g_\delta}u_{\delta}^2\right)dv_{g_\delta}\
-\int_{\partial M}2H_{g_\delta}^{-}u_{\delta}^2 d\sigma_{g_\delta}.
\end{align*}
In particular, 
\begin{align*}
m_{\partial M}(g)-m_{\partial M}(\widetilde g_\delta)
&=\int_{M}\left( \frac{4(n-1)}{n-2}|\nabla_{g_\delta}u_\delta|_{g_\delta}^2-R^-_{g_\delta}u_{\delta}^2\right)dv_{g_\delta}\
-\int_{\partial M}2H_{g_\delta}^{-}u_{\delta}^2 d\sigma_{g_\delta}%\\
%&\leq\int_{M} \frac{4(n-1)}{n-2}|\nabla_{g_\delta}w_\delta|_{g_\delta}^2+C\left(\int_{M}R^-_{g_\delta}dv_{g_\delta}\
%+\int_{\partial M}2H_{g_\delta}^{-} d\sigma_{g_\delta}\right)
\end{align*}
which goes to zero as $\delta\to 0$ according to Proposition \ref{propo:lim:u} and inequality \eqref{ineq:w}. 
\end{proof}

As stated just above Proposition \ref{propo:conv:mass}, that proposition implies that $m_{\partial M}(g)\geq 0$. This proves the first part of Theorem \ref{main:thm}.

%%%%%%%%%%%%%%%%%%%%%%%%%%%%%%%%%%%%%%%%%%%%%%%%%%%%%%%%%%%%%%%%%%%%%%%

\subsection{The case of strict positivity of the mass}\label{subsec:positive}

In this subsection, we assume that $H^\Sigma_{g_-}>H^\Sigma_{g_+}$ at some point on $\Sigma$. Let us prove that $m_{\partial M}(g)>0$. 

Since $H^\Sigma_{g_-}$ and $H^\Sigma_{g_+}$ are continuous on $\Sigma$, we can choose a compact subset $K\subset \Sigma$ and some $\lambda>0$ such that
\begin{equation*}
H^\Sigma_{g_-}(x)-H^\Sigma_{g_+}(x)\geq \lambda,\: \forall x\in K.
\end{equation*} 
Therefore, by Proposition \ref{propo:control:curv}, we have
\begin{equation*}
R_{g_{\delta}}^+(x,t)\geq \lambda \frac{100}{\delta^2}\chi\left(\frac{100t}{\delta^2}\right)-C_0, \:\forall (x,t)\in K\times \left[-\frac{\delta^2}{100}, \frac{\delta^2}{100}\right].
\end{equation*}

Since $R_{\widetilde{g_{\delta}}}\geq 0$, $H^{\Sigma}_{\widetilde{g_{\delta}}}\geq 0$, and $\widetilde{g_{\delta}}$ is asymptotically flat of order $\tau>0$, there exists a positive solution $v_\delta$ to the equations
\begin{align*}
	\begin{cases}
		\Delta_{\widetilde{g_{\delta}}} v_{\delta}-c_nR_{\widetilde{g_{\delta}}}v_{\delta}=0,&M,\\
		\displaystyle \frac{\d  v_{\delta}}{\d \eta_{\widetilde g_\delta}}+2c_nH^{\d M}_{\widetilde{g_{\delta}}}v_{\delta}=0, &\partial M,
	\end{cases}
\end{align*}
such that $v_\delta\to 1$ as $|x|\to \infty$.  Indeed, we write $v_\delta=1+z_\delta$ and obtain a solution $z_\delta\in C^{2,\alpha}_{-\tau}(M)$ to
\begin{align}\label{eq:v_delta}
	\begin{cases}
		\Delta_{\widetilde{g_{\delta}}} z_{\delta}-c_nR_{\widetilde{g_{\delta}}}z_{\delta}=c_nR_{\widetilde{g_{\delta}}},&M,\\
		\displaystyle \frac{\d  z_{\delta}}{\d \eta_{\widetilde g_\delta}}+2c_nH^{\d M}_{\widetilde{g_{\delta}}}z_{\delta}=-2c_nH^{\d M}_{\widetilde{g_{\delta}}}, &\partial M,
	\end{cases}
\end{align}
as follows. Define $\widetilde T_\delta w=\left(\Delta_{\widetilde{g_{\delta}}} w-c_nR_{\widetilde{g_{\delta}}}w, \d  w/\d \eta_{\widetilde g_\delta}+2c_nH^{\d M}_{\widetilde{g_{\delta}}}w\right)$ the operator
$$
\widetilde T_\delta:C^{2,\alpha}_{-\tau}(M)\to C^{0,\alpha}_{-\tau-2}(M)\times C^{1,\alpha}_{-\tau-1}(\d M).
$$
As $R_{\widetilde{g_{\delta}}}\in C^{0,\alpha}_{-\tau-2}(M)$, $H^{\Sigma}_{\widetilde{g_{\delta}}}\in C^{1,\alpha}_{-\tau-1}(\d M)$, it follows from Proposition 3.3 in ~\cite{almaraz-barbosa-lima} that the $\widetilde T_{\delta}$ are isomorphisms. 

By the maximum principle, we have that 
$$0<v_{\delta}\leq 1.$$ Now, let us define $$\widehat{g_{\delta}}=v_{\delta}^{\frac{4}{n-2}}\widetilde{g_{\delta}}.$$
Similar to previous discussions, we know that $\widehat{g_{\delta}}$ is an asymptotically flat metric with zero scalar curvature and zero mean curvature on the boundary. Moreover, the mass $m_{\d M}(\widehat{g_{\delta}})$ and $m_{\d M}(\widetilde{g_{\delta}})$ are related by 
$$m_{\d M}(\widetilde{g_{\delta}})=m_{\d M}(\widehat{g_{\delta}})+\int_{M}\left( \frac{4(n-1)}{n-2}|\nabla_{\widetilde{g_{\delta}}}v_\delta|_{\widetilde{g_{\delta}}}^2+R_{\widetilde{g_{\delta}}}v_{\delta}^2\right)dv_{\widetilde{g_{\delta}}}\
+\int_{\partial M}2H_{\widetilde{g_{\delta}}}^{\partial M}v_{\delta}^2 d\sigma_{\widetilde{g_{\delta}}},$$
where $m_{\d M}(\widehat{g_{\delta}})\geq 0$ by the positive mass theorem in \cite{almaraz-barbosa-lima}. Now, one can argue similarly as in Proposition 4.2 of \cite{miao} to show that
$$\inf_{\delta>0}\left[\int_{M}\left( \frac{4(n-1)}{n-2}|\nabla_{\widetilde{g_{\delta}}}v_\delta|_{\widetilde{g_{\delta}}}^2+R_{\widetilde{g_{\delta}}}v_{\delta}^2\right)dv_{\widetilde{g_{\delta}}}\
+\int_{\partial M}2H_{\widetilde{g_{\delta}}}^{\partial M}v_{\delta}^2 d\sigma_{\widetilde{g_{\delta}}}\right]>0. $$
This shows that $\inf_{\delta>0}m_{\d M}(\widetilde{g_{\delta}})>0$, therefore $m_{\d M}(g)>0$, and proves the second part of Theorem \ref{main:thm}.

%%%%%%%%%%%%%%%%%%%%%%%%%%%%%%%%%%%%%%%%%%%%%%%%%%%%%%%%%%%%%%%%%%%%%%%%%%%%%%%%%%%%%

\section{Application to the positive mass theorem on manifolds with non-compact corner}
\label{sec:application}

In this section, as an application of Theorem \ref{main:thm} we will prove Theorem \ref{appl:thm}. Let us first present some definitions and notations involving cornered spaces and related mass quantity. 

A {\it{cornered space}} \cite{mckeown} is an $n$-dimensional smooth Riemannian manifold $(M,g)$ with codimension two
corners such that 

\vspace{0.2cm}
(i) there are submanifolds with boundary $\Sigma_1\subset \partial M$ and $\Sigma_2\subset \partial M$ of the boundary $\partial M$, such that $\Gamma=\Sigma_1\cap\Sigma_2$ is the mutual boundary, and is the entire codimension-two corner of $M$, and such that $\partial M=\Sigma_1\cup\Sigma_2$;

\vspace{0.2cm}
\noindent and

(ii) the corner $\Gamma\subset M$ is a smooth submanifold of $M$.

\vspace{0.2cm}
We denote a cornered space by $(M,\Sigma_1, \Sigma_2, g)$.
Let $\eta_1$ and $\eta_2$ be the outward pointing unit normal vectors to $\Sigma_1$ and $\Sigma_2$ respectively and consider the mean curvatures  $H_g^{\Sigma_1}=\text{div}_g\eta_1$ and $H_g^{\Sigma_2}=\text{div}_g\eta_2$.

Recall that  $\mathbb R^n_{\mathbb L}=\{x=(x_1,...,x_n)\in\mathbb R^n\: ;\: x_1\geq 0,\, x_n\geq 0\}$
is the Euclidean quarter-space. 
We assume that the cornered space $(M,\Sigma_1, \Sigma_2, g)$ with dimension $n\geq 3$ is {\it{asymptotically flat (with a non-compact corner)}} in the sense that there is a compact subset $K\subset M$ such that 
$M\backslash K$ is diffeomorphic to $\{x\in\mathbb R^n_{\mathbb L}\: ;\:|x|> 1\}$, and in the induced coordinates $g$ satisfies 
$$
\sum_{a,b,c,d=1}^{n}\left( 
|g_{ab}(x)-\delta_{ab}|+|x||g_{ab,c}(x)|+|x|^2|g_{ab,cd}(x)|
\right)
\leq C|x|^{-\tau},
$$
for some $C>0$ and $\tau>(n-2)/2$, and 
$$
\sum_{a=1}^{n}\frac{x_a}{|x|}g_{1a}+\sum_{a=1}^{n}\frac{x_a}{|x|}g_{na}
$$
is integrable on $\Gamma$. Recall that the comas stand for partial derivatives.
Assume further that the scalar curvature $R_g$ is integrable on $M$ and the mean curvatures $H_g^{\Sigma_1}$ and $H_g^{\Sigma_2}$ are integrable on $\Sigma_1$ and $\Sigma_2$ respectively.
We define the mass of $(M,\Sigma_1, \Sigma_2, g)$ by
\begin{align}\label{mass:corner}
m_{\Sigma_1,\Sigma_2}(g)=\lim_{\rho\to\infty}
&\Big(
\sum_{a,b=1}^n\int_{x\in\mathbb R^n_{\mathbb L},\,|x|=\rho} (g_{ab,b}-g_{bb,a})\frac{x_a}{|x|}
\\
&+\sum_{i=2}^{n}\int_{x\in\mathbb R^n_{\mathbb L},\,|x|=\rho,\,x_1=0} g_{1i}\frac{x_i}{|x|}
+\sum_{i=1}^{n-1}\int_{x\in\mathbb R^n_{\mathbb L},\,|x|=\rho,\,x_n=0} g_{ni}\frac{x_i}{|x|}
\Big)\notag
\end{align}
where we are omitting the area elements induced by the Euclidean metric.

\begin{remark}
The definition of $m_{\Sigma_1,\Sigma_2}(g)$ is introduced in \cite{almaraz-lima-mckeown} where it is also proved that the limit in \eqref{mass:corner} exists and is finite, and it is independent of the choice of asymptotic coordinates $x=(x_1,...,x_n)$. Moreover, combined with our Theorem~\ref{appl:thm}, one can obtain the rigidity part of the positive mass theorem for our setting. We will not explore this in the current manuscript. 
\end{remark}

The rest of this section is devoted to the proof of Theorem \ref{appl:thm}. First of all, we need the following:

\begin{proposition}\label{prop:conf:flat}
	Let $(M,\Sigma_1, \Sigma_2,g)$ be an asymptotically flat manifold, with a non-compact corner, of order $\tau>(n-2)/2$ , where $R_g\geq 0$, $H^{\Sigma_1}_g\geq 0$ and $H^{\Sigma_2}_g\geq 0$. Then for any small $\epsilon>0$ there exists an asymptotic flat metric $g_\epsilon$ of order $\tau-\epsilon$ such that:

	(i)  $R_{g_\epsilon}\geq 0$, $H^{\Sigma_1}_{g_\epsilon}\geq 0$ and $H^{\Sigma_2}_{g_\epsilon}\geq 0$ everywhere, with $R_{g_\epsilon}= 0$, $H^{\Sigma_1}_{g_\epsilon}=0$ and $H^{\Sigma_2}_{g_\epsilon}=0$ outside a compact set K;
	
	(ii) $g_\epsilon$ is conformally flat outside K;
	
	(iii) $\left|m_{\Sigma_1,\Sigma_2}(g_{\epsilon})-m_{\Sigma_1,\Sigma_2}(g)\right|\leq \epsilon$.
\end{proposition}

The proof of Proposition \ref{prop:conf:flat} follows the same lines as \cite[Proposition 4.1]{almaraz-barbosa-lima} (see also  \cite[Lemma 10.6]{lee-parker}) but will be presented below for convenience. The main tool is the next proposition which plays the same role of Proposition 3.3 in \cite{almaraz-barbosa-lima} and whose proof is left to the appendix section.

\begin{proposition}\label{prop:isom}
	Let $(M,\Sigma_1, \Sigma_2,g)$ be an asymptotically flat manifold, with a non-compact corner. Fix $2-n<\gamma<0$ and let 
$$T:C^{1,\alpha}_{\gamma}(M)\cap C^{2,\alpha}_{\gamma}(M\backslash \Gamma)\rightarrow C^{0,\alpha}_{\gamma-2}(M\backslash\Gamma)\times C^{1,\alpha}_{\gamma-1}(\Sigma_1\backslash\Gamma)\times C^{1,\alpha}_{\gamma-1}(\Sigma_2\backslash\Gamma)$$  
be defined by 
$$T(u)=(-\Delta_g u+hu, \d u/\d \eta_1+h_1u, \d u/\d \eta_2),$$  
where $h\in C^{0,\alpha}_{-2-\epsilon}(M)$ and $h_1\in C^{1,\alpha}_{-1-\epsilon}(\Sigma_1)$ for some small $\epsilon>0$. If $h\geq 0$ and $h_1\geq 0$, then $T$ is an isomorphism. 
\end{proposition}

\begin{remark}
In this paper, we only use the special case where both $h$ and $h_1$ vanish in Proposition \ref{prop:isom}. 
\end{remark}

\begin{proof}[Proof of Proposition \ref{prop:conf:flat}]
Let $\chi:\R\to [0,1]$ be a smooth cut-off function such that $\chi(t)=1$ for $t\leq 1$ and $\chi(t)=0$ for $t\geq 2$. 
For $R>0$ large define $\chi_R(x)=\chi(R^{-1}r(x))$ and set $g_R=\chi_R g+(1-\chi_R)g_0$, where $g_0$ stands for the Euclidean metric. 
We will solve
\begin{equation}\label{sistema:uR}
\begin{cases}
L_{g_R}u_R=\chi_R R_g u_R&\text{in}\:M\backslash\Gamma,
\\
B_{g_R}u_R=\chi_R H_g u_R&\text{on}\:(\Sigma_1\cup\Sigma_2)\backslash\Gamma,
\end{cases}
\end{equation}
for $u_R>0$ and $R$ large enough, and check that the conformal metric $\overline g_R=u_R^{\frac{4}{n-2}}g_R$ has all the desired properties.
We write $u_R=1+v_R$ and set 
$$
L_R=-\Delta_{g_R}+\frac{n-2}{4(n-1)}\gamma_R, \quad
B_R=\frac{\d}{\d\eta_{g_R}}+\frac{n-2}{2(n-1)}\overline{\gamma}_R,
$$
where $\gamma_R=R_{g_R}-\chi_R R_g$ and $\overline{\gamma}_R=H_{g_R}-\chi_R H_g$.  
Observe that $H_g|_{\Sigma_i\backslash\Gamma}=H_g^{\Sigma_i}$, $i=1,2$. 
Thus, (\ref{sistema:uR}) is equivalent to
\begin{equation}\label{sistema:psiR}
\begin{cases}
\frac{4(n-1)}{n-2} L_Rv_R=-\gamma_R&\text{in}\:M\backslash\Gamma,
\\
\frac{2(n-1)}{n-2} B_Rv_R=-\overline{\gamma}_R&\text{on}\:(\Sigma_1\cup\Sigma_2)\backslash\Gamma.
\end{cases}
\end{equation}
For any $\e>0$ we have $\|\gamma_R\|_{C^{0,\a}_{-\tau-2+\e}(M\backslash\Gamma)}\to 0$ and  $\|\overline{\gamma}_R\|_{C^{1,\a}_{-\tau-1+\e}((\Sigma_1\cup\Sigma_2)\backslash\Gamma)}\to 0$ as $R\to \infty$. 
In what follows, we solve (\ref{sistema:psiR}) uniquely for $v_R\in C^{2,\a}_{-\tau+\e}(M\backslash\Gamma)\cap C^{1,\a}_{-\tau+\e}(M)$, with $\|v_R\|_{C^{2,\a}_{-\tau+\e}(M\backslash\Gamma)}\to 0$ and $\|v_R\|_{C^{1,\a}_{-\tau+\e}(M)}\to 0$ as $R\to \infty$. 

Fix $0<\e<\tau-\frac{n-2}{2}$.  According to Proposition \ref{prop:isom}, the operator
$$
T:C^{2,\a}_{-\tau+\e}(M\backslash\Gamma)\cap C^{1,\a}_{-\tau+\e}(M)\to C^{0,\a}_{-\tau+\e-2}(M\backslash\Gamma)\times C^{1,\a}_{-\tau+\e-1}((\Sigma_1\cup\Sigma_2)\backslash\Gamma)
$$ 
defined by $T u=(\Delta_gu,\d u/\d\eta_g)$ is an isomorphism, where $\eta_g|_{\Sigma_i\backslash\Gamma}=\eta_g^{\Sigma_i}$, $i=1,2$.
Here, we are seeing the intersection of spaces $C^{2,\a}_{-\tau+\e}(M\backslash\Gamma)$ and $C^{1,\a}_{-\tau+\e}(M)$ with the sum of their two respective norms. 
Set $\displaystyle T_Ru=(L_Ru,B_Ru)$.
It follows from the directly checked estimates 
$$
\|(\Delta_{g_R}-\Delta_g)u)\|_{ C^{0,\a}_{-\tau+\e-2}(M\backslash\Gamma)}\leq C\|g_R-g\|_{ C^{1,\a}_{0}(M\backslash\Gamma)}\|u\|_{ C^{2,\a}_{-\tau+\e}(M\backslash\Gamma)}\;,
$$
$$
\big\|\frac{\d u}{\d\eta_{g_R}}-\frac{\d u}{\d\eta_{g}}\big\|_{ C^{1,\a}_{-\tau+\e-1}((\Sigma_1\cup\Sigma_2)\backslash\Gamma)}\leq 
C\|g_R-g\|_{ C^{1,\a}_{0}(M\backslash\Gamma)}\|u\|_{ C^{2,\a}_{-\tau+\e}(M\backslash\Gamma)}\;,
$$
$$
\|\gamma_R u\|_{ C^{0,\a}_{-\tau+\e-2}(M\backslash\Gamma)}\leq 
C\|\gamma_R\|_{ C^{0,\a}_{-2}(M\backslash\Gamma)}\|u\|_{ C^{0,\a}_{-\tau+\e}(M)}\;,
$$
and
$$
\|\overline{\gamma}_R u\|_{ C^{1,\a}_{-\tau+\e-1}((\Sigma_1\cup\Sigma_2)\backslash\Gamma)}\leq 
C\|\overline{\gamma}_R\|_{ C^{1,\a}_{-1}((\Sigma_1\cup\Sigma_2)\backslash\Gamma)}\|u\|_{ C^{1,\a}_{-\tau+\e}(M)}
$$
for some constant $C>0$, that $T_R-T$ is arbitrarily small in the operator norm as $R\to\infty$. From this we conclude that $T_R$ is also an isomorphism for large $R$, which provides a unique solution $v_R$ to (\ref{sistema:psiR}). 

Now we can choose $g_\epsilon=\overline{g}_R$ for $R$ large, proving (i) and (ii).
It is easy to prove that $\overline g_R\to g$ in the space of asymptotically flat metrics with order $\tau-\epsilon$ in the sense that 
$$\|\overline{g}_R-g\|_{C^{1,\alpha}_{-\tau+\epsilon}}(M)\to 0,$$ and 
$$\|R_{\overline{g}_R}-R_g\|_{L^1(M)}+\|H_{\overline{g}_R}-H_g\|_{L^1(\d M)}\to 0 $$ as $R\to\infty$, so that the property (iii) also holds.
\end{proof}

Now we are ready to prove Theorem \ref{appl:thm} by using a doubling argument.

Observe that, under the asymptotic coordinates, $\Sigma_1$ and $\Sigma_2$ correspond to the hypersurfaces $\{x_1=0\}$ and $\{x_n=0\}$, respectively. Hence, we can assume that $\Sigma_1$ and $\Sigma_2$ are umbilic outside $K$ with respect to the conformally flat metric $g_\epsilon$. Since $\Sigma_1$ and $\Sigma_2$ are umbilic satisfying $H^{\Sigma_1}_{g_\epsilon}=H^{\Sigma_2}_{g_\epsilon}=0$ outside $K$, it follows that $\Sigma_1$ and $\Sigma_2$ are totally geodesic outside $K$ with respect to $g_\epsilon$.

The idea now will be to double $M$ along $\Sigma_1$ or $\Sigma_2$ to obtain a manifold in the setting of Theorem \ref{main:thm}. Due to notation convenience, we choose to double along $\Sigma_2$. Precisely, we consider the manifold $\widetilde M=M\times \{0,1\}/\sim$, where $(p,0)\sim (p,1)$ if and only if $p\in\Sigma_2$, and the metric $\widetilde g_\epsilon\left([(p,j)]\right)=g_\epsilon(p)$ for all $p\in M$, $j=0,1$. We claim that $\widetilde g_\epsilon$ is of class $C^{2,\alpha}$ on $\widetilde M\backslash \widetilde K$, where $\widetilde K$ is the doubling of $K$, i.e., $\widetilde K=\left\{[(p,j)]\,;\:p\in K, j=0,1\right\}$.

Before proceeding, we shall clarify the topological smooth structure given to $\widetilde M$ near $\Gamma=\Sigma_1\cap\Sigma_2$. Let $p\in\Gamma$ and consider any local coordinates 
$$
\Phi:U\subset M\to\mathbb R^n_{\mathbb L},\:\:\Phi(p)=0,
$$
where $U$ is an open neighbourhood of $p$. Choosing $U$ maybe smaller, we can assume that
$$
\Phi(U)=\{z\in \mathbb R^n_{\mathbb L}\,;\:|z|<1\}.
$$
A coordinate chart on $\widetilde M$ around the equivalence class $[(p,0)]=[(p,1)]\in\widetilde M$ is obtained in the following way. We define
$$
\Psi: \widetilde U\to\mathbb \{z\in \mathbb R^n_+\,;\:|z|<1\}
$$
on $\widetilde U=\left\{[(q,j)]\in\widetilde M\,;\:q\in U, \:j=0,1 \right\}$ by 
$$
\Psi([q,0])=\Phi(q)\qquad\text{and}\qquad \Psi([q,1])=\widetilde{\Phi(q)},
$$
where we are denoting $\widetilde z=(z_1,...,z_{n-1}, -z_n)$ for $z=(z_1,...,z_n)$.
With that structure, $\widetilde M$ is a smooth manifold with (smooth) boundary $\partial\widetilde M$ given by two copies of $\Sigma_1$ glued along its boundary $\Gamma$.

Let us prove that $\widetilde{g}_\epsilon\in C^{2,\alpha}\left(\widetilde M\backslash(\partial\widetilde M\cup(\Sigma_2\cap K))\right)\cap  C^{1,\alpha}(\widetilde M\backslash(\Sigma_2\cap K))$. 
We fix a chart $\Phi$ as above such that the coordinates are Fermi with respect to $\Sigma_2$ and the domain of $\Phi$ is contained in $M\backslash K$. It means that the coordinate curves $t\rightarrow \Phi^{-1}(x_1,...,x_{n-1},t)$ are unit geodesics of $M$ orthogonal to $\Sigma_2$. The fact that $\Sigma_1$ is totally geodesic outside $K$ ensures that those curves remain tangent to $\Sigma_1$ when $x_1=0$. Thus in the coordinates induced by $\Phi$, $(g_\epsilon)_{an}=\delta_{an}$, $a=1,...,n$, and the fact that $\Sigma_2$ is totally geodesic implies that $(g_\epsilon)_{ij,n}=0$ on $x_n=0$ for $i,j=1,..., n-1$. 
As one can check, this ensures that $g$ is of class $C^{2,\alpha}$ across the interior of $\Sigma_2$ outside $K$. This implies the claimed regularity for $\widetilde{g}_\epsilon$.

We claim that $m_{\partial\widetilde M}(\widetilde g_\epsilon)\geq 0$ so we need to make sure the hypotheses of Theorem \ref{main:thm} are fulfilled. 
Extend the hypersurface $\Sigma_2\cap K$ to a compact hypersurface $\Sigma$ with boundary $\partial \Sigma= \Sigma\cap\partial M$ in such a way that $\Sigma$ is orthogonal to $\partial M$ (see Figure \ref{picture}). The open subset $\Omega\subset \widetilde M$ satisfying $\partial \Omega=\Sigma\cup (\Omega\cap\partial M)$ is such that  
$$
\widetilde g^-_\epsilon:=\widetilde g_\epsilon|_\Omega\in C^{1,\alpha}(\Omega)\cap C^{2,\alpha}(\text{int}\,\Omega),
$$
$$
\widetilde g^+_\epsilon:=\widetilde g_\epsilon|_{M\backslash\overline\Omega}\in C^{1,\alpha}(M\backslash\overline\Omega)\cap C^{2,\alpha}(\text{int}\,(M\backslash\overline\Omega)),
$$
and $\widetilde g^-_\epsilon$ and $\widetilde g^+_\epsilon$ are $C^2$ up to $\Sigma\backslash \partial\Sigma$.

In order to verify that $H^\Sigma_{\widetilde g^-_\epsilon}\geq H^\Sigma_{\widetilde g^+_\epsilon}$, observe that 
$$
H^\Sigma_{\widetilde g^-_\epsilon}\Big |_{\Sigma_2\cap K}=H^{\Sigma_2}_{g_\epsilon}\Big |_{\Sigma_2\cap K}\geq 0
$$
and 
$$
H^\Sigma_{\widetilde g^+_\epsilon}\Big |_{\Sigma_2\cap K}=-H^{\Sigma_2}_{g_\epsilon}\Big |_{\Sigma_2\cap K}.
$$
On the other hand, since $\widetilde g_\epsilon$ is $C^{2,\alpha}$ across $\Sigma$ outside $K$, $H^\Sigma_{\widetilde g^-_\epsilon}$ and $H^\Sigma_{\widetilde g^+_\epsilon}$ coincide on that region.

\begin{figure}
  \includegraphics[width=\linewidth]{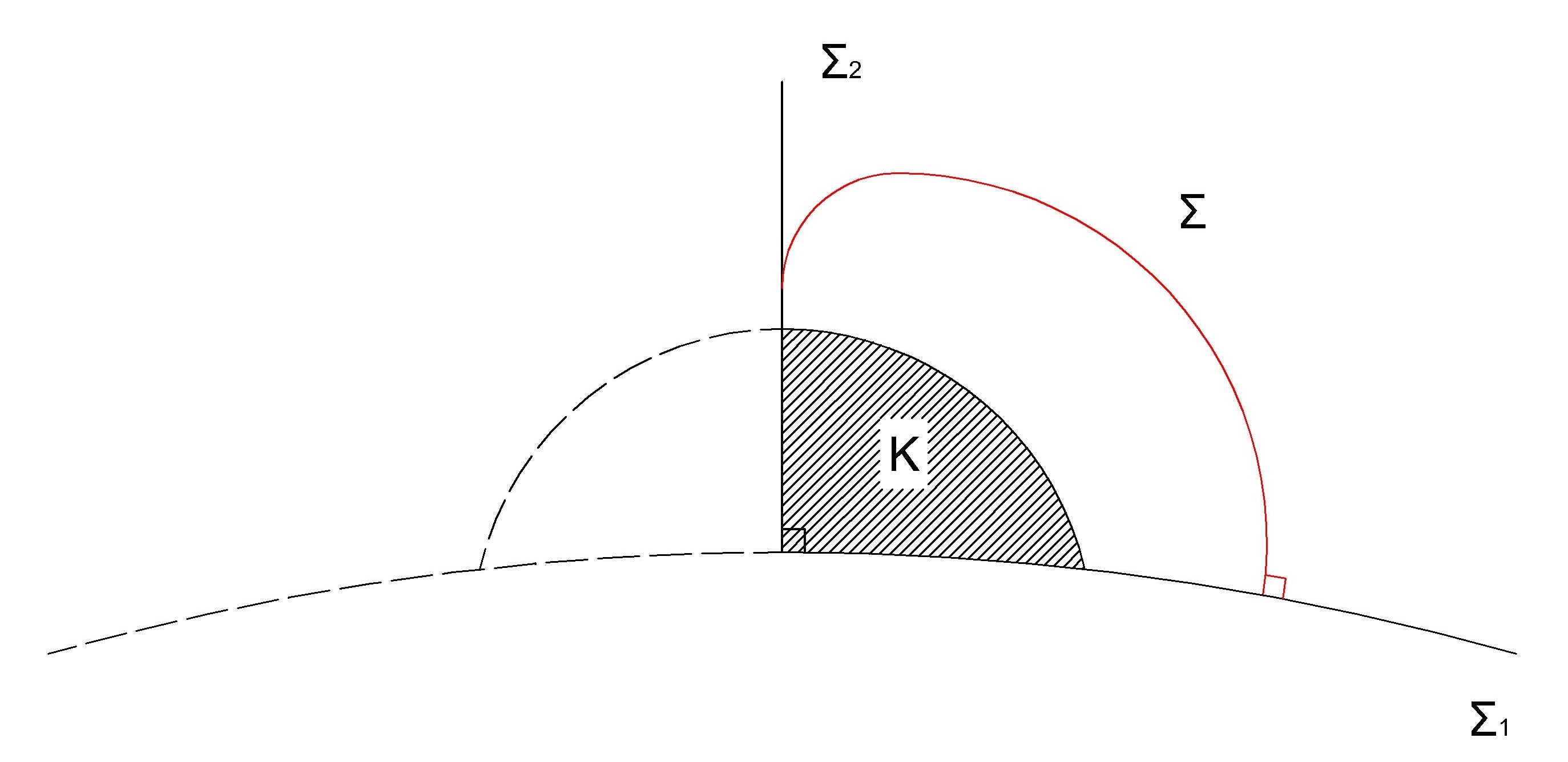}
  \caption{The hypersurface $\Sigma$.}
  \label{picture}
\end{figure}

So we conclude that $m_{\partial M}(\widetilde g_\epsilon)\geq 0$. 
To finish the proof of Theorem \ref{appl:thm} we need to relate the mass $m_{\Sigma_1, \Sigma_2}(g_\epsilon)$ of $M$ with the mass $m_{\partial M}(\widetilde g_\epsilon)$ of $\widetilde M$.
Since $\Sigma_1$ and $\Sigma_2$ are conformally Euclidean outside a compact set with respect to $g_\epsilon$, the last two terms in \eqref{mass:corner} vanish so that
$$
m_{\Sigma_1,\Sigma_2}(g_\epsilon)=\lim_{\rho\to\infty}
\sum_{a,b=1}^n\int_{x\in\mathbb R^n_{\mathbb L},\,|x|=\rho} ((g_\epsilon)_{ab,b}-(g_\epsilon)_{bb,a})\frac{x_a}{|x|}.
$$
On the other hand,
$$
m_{\partial\widetilde M}(\widetilde g_\epsilon)=\lim_{\rho\to\infty}
\sum_{a,b=1}^n\int_{x\in\mathbb R^n_{+},\,|x|=\rho} ((\widetilde g_\epsilon)_{ab,b}-(\widetilde g_\epsilon)_{bb,a})\frac{x_a}{|x|},
$$
and clearly $m_{\partial\widetilde M}(\widetilde g_\epsilon)=2m_{\Sigma_1,\Sigma_2}(g_\epsilon)$. This proves Theorem \ref{appl:thm}.

%%%%%%%%%%%%%%%%%%%%%%%%%%%%%%%%%%%%%%%%%%%%%%%%%%%%%%%%%%%%%%%%%%%%%%%%%%%%%%%%%%%%%%%%%%%%%

\appendix

%%%%%%%%%%%%%%%%%%%%%%%%%%%%%%%%%%%%%%%%%%%%%%%%%%%%%%%%%%%%%%%%%%%%%%%%%%%%%%%%%%%%%%%

\section{Appendix}

In this appendix, we give a proof of Proposition \ref{prop:isom} which is inspired by the case of manifolds with smooth boundary in \cite[Appendix A]{almaraz-barbosa-lima}. It builds on a series of lemmas for which we only give the critical parts of proofs.

Let us first state an existence and regularity result which is needed in the proof of the next lemmas. Let $(N,\Sigma_1',\Sigma_2', g')$ be a compact cornered space (see Section \ref{sec:application}). Denote by $\Gamma'$ the corner $\Sigma_1'\cap\Sigma_2'$ and assume that $\Sigma_1'$ and $\Sigma_2'$ are orthogonal along $\Gamma'$.  
Set $Lu=\Delta_g+h$ and $B_iu=\partial u/\partial \eta_i+h_iu$, $i=1,2$, where $h\in C^{0,\alpha}(N)$ and $h_i\in C^{1,\alpha}(\Sigma_i')$ for some $0<\alpha<1.$
Consider the problem
\begin{equation}\label{eq:cpct:problem}
	\begin{cases}
		Lu=f,&N\backslash\Gamma',
		\\
		B_1u=f_1, &\Sigma_1'\backslash\Gamma',
		\\
		B_2u=f_2, &\Sigma_2'\backslash\Gamma',
	\end{cases}
\end{equation}
where $f\in C^{0,\alpha}(N\backslash\Gamma')$ and $f_i\in C^{1,\alpha}(\Sigma_i'\backslash\Gamma')$.
\begin{proposition}\label{propo:exist:cpct}
	Suppose there is no non-trivial solution $u$ to \eqref{eq:cpct:problem} with $f=0$ and $f_i=0$, $i=1,2$. Then, given $f\in C^{0,\alpha}(N\backslash\Gamma')$ and $f_i\in C^{1,\alpha}(\Sigma_i'\backslash\Gamma')$, $i=1,2$, there is a solution $u\in C^{2,\alpha}(N\backslash \Gamma')\cap C^{1,\alpha}(N)$ to \eqref{eq:cpct:problem} and there is $C=C(N,\Sigma'_1,\Sigma'_2, g')$ satisfying
$$
\|u\|_{C^{2,\alpha}(N\backslash\Gamma')}+\|u\|_{C^{1,\alpha}(N)}\leq
C(\|f\|_{C^{0,\alpha}(N\backslash\Gamma')}+\|f_1\|_{C^{1,\alpha}(\Sigma'_1\backslash\Gamma')}+\|f_2\|_{C^{1,\alpha}(\Sigma'_2\backslash\Gamma')}).
$$
\end{proposition}

\begin{proof}
	See for example Theorem 3.2 in \cite{lieberman2} together with the embedding result on the first page of \cite{lieberman}.
\end{proof}

Now let us introduce the lemmas to be used.
	\begin{lemma}\label{lemma:A1} Set $\Gamma=\{x\in\partial \mathbb R^n_{\mathbb L}\,;\:x_1=x_n=0\}$. 
		If $2-n<\gamma<0$, there exists $C=C(n)>0$ such that for all $u\in C^2_{\gamma}(\R^n_{\mathbb L}\backslash\Gamma)\cap C^1_{\gamma}(\R^n_{\mathbb L})$, we have 
		$$\|u\|_{C_{\gamma}^0(\R^n_{\mathbb L})}\leq C\(\|\Delta u\|_{C_{\gamma-2}^0(\R^n_{\mathbb L}\backslash\Gamma)}+\|\d u/\d x_n\|_{C_{\gamma-1}^0(\d \R^n_{\mathbb L}\cap \{x_n= 0\})}+\|\d u/\d x_1\|_{C_{\gamma-1}^0(\d \R^n_{\mathbb L}\cap \{x_1= 0\})}\).$$
	\end{lemma}

\begin{proof}
For $x=(x_1,...,x_n)$ and $y=(y_1,...,y_n)$ in $\mathbb R^n$ we set $\widetilde y=(y_1,...,y_{n-1},-y_n)$ and $y'=(-y_1,y_2,...,y_n)$, and define
$$
G(x,y)=\psi(x,y)+\psi(x,y')\,,
\qquad 
\psi(x,y)=|x-y|^{2-n}+|x-\widetilde y|^{2-n}.
$$ 
Observe that for any $x,y\in\mathbb R^n_{\mathbb L}$ with $x\neq y$ we have $\Delta_x G(x,y)=0$, 
$$
\frac{\partial}{\partial x_1}G(x,y)=0, \:\text{if}\:x_1=0
$$
and
$$
\frac{\partial}{\partial x_n}G(x,y)=0, \:\text{if}\:x_n=0.
$$
We proceed as in \cite[Lemma A.1]{almaraz-barbosa-lima} by using Green's formula to arrive at
\begin{align}\label{eq:Green}
(2-n)\omega_{n-1}u(y)=&\int_{x\in\mathbb R^n_{\mathbb L}}G(x,y)\Delta u(x)\,dx
\\
&\hspace{0.2cm}+\int_{x\in\mathbb\partial R^n_{\mathbb L},\, x_1=0}G(x,y)\frac{\partial u}{\partial x_1}(x)\,d\sigma(x)\notag
\\
&\hspace{0.4cm}+\int_{x\in\mathbb\partial R^n_{\mathbb L},\, x_n=0}G(x,y)\frac{\partial u}{\partial x_n}(x)\,d\sigma(x),\notag
\end{align}
where $\omega_{n-1}$ is the area of the unit sphere in $\mathbb R^n$.
Since $2-n<\gamma$, we can use the fact that
$$
\int_{x\in\mathbb R^n_{\mathbb L}}|x-y|^{2-n}|x|^{\gamma-2}dx
+\int_{x\in\mathbb\partial R^n_{\mathbb L}}|x-y|^{2-n}|x|^{\gamma-1}d\sigma(x)\leq C(n)|y|^\gamma
$$
for any $y\in\mathbb R^n$, so it follows from \eqref{eq:Green} that 
$$
|y|^{-\gamma}|u(y)|\leq C\|\Delta u\|_{C^0_{\gamma-2}\left(\mathbb R^n_{\mathbb L}\backslash\Gamma\right)}
+C\left\|\frac{\partial u}{\partial x_1}\right\|_{C^0_{\gamma-1}\left(\mathbb\partial R^n_{\mathbb L},\, x_1=0\right)}
+C\left\|\frac{\partial u}{\partial x_n}\right\|_{C^0_{\gamma-1}\left(\mathbb\partial R^n_{\mathbb L},\, x_n=0\right)}
$$
which proves the lemma.
\end{proof}

\begin{lemma}\label{lemma:A2}
	Let $(M,\Sigma_1,\Sigma_2,g)$ be an asymptotically flat manifold, with a non-compact corner $\Gamma$, as defined in Section \ref{sec:application} and $\gamma\in \R$. Then the following holds: \\
	(a) There exists $C=C(M,\Sigma_1,\Sigma_2,g,\gamma)>0$ such that if $u\in C^{1,\alpha}_{\gamma}(M\backslash\Gamma)$, $\Delta_g u\in C^{0,\alpha}_{\gamma-2}(M\backslash\Gamma)$ and $\d u/ \d \eta_i\in C^{1,\alpha}_{\gamma-1}(\Sigma_i\backslash\Gamma)$, $i=1,2$, then $u\in C^{1,\alpha}_{\gamma}(M)\cap C^{2,\alpha}_{\gamma}(M\backslash \Gamma)$ and we have 
	$$\|u\|_{C^{2,\alpha}_{\gamma}(M\backslash \Gamma)}+\|u\|_{C^{1,\alpha}_{\gamma}(M)}
		\leq C\(\|\Delta_g u\|_{C^{0,\alpha}_{\gamma-2}(M\backslash\Gamma)}+\sum_{i=1}^{2}\|\d u/\d \eta_i\|_{C^{1,\alpha}_{\gamma-1}(\Sigma_i\backslash\Gamma)}+\|u\|_{C^{0}_{\gamma}(M)}\).$$ \\
	(b) Assume that $g$ coincides with the Euclidean metric outside a compact set and $2-n<\gamma <0$. Then there exists $C=C(M, \Sigma_1, \Sigma_2, g,\gamma)>0$ and a compact set $K\subset M$ such that for all $u\in C^{1,\alpha}_{\gamma}(M)\cap C^{2,\alpha}_{\gamma}(M\backslash \Gamma)$ 
$$
\|u\|_{C^{0}_{\gamma}(M)}\leq C\(\|\Delta_g u\|_{C^{0}_{\gamma-2}(M\backslash\Gamma)}+\sum_{i=1}^{2}\|\d u/\d \eta_i\|_{C^{0}_{\gamma-1}(\Sigma_i\backslash\Gamma)}+\|u\|_{C^{1}(K)}\).
$$ 
In particular, if $u\in C^{1,\alpha}_{\gamma}(M\backslash\Gamma)$, then $u\in C^{1,\alpha}_{\gamma}(M)\cap C^{2,\alpha}_{\gamma}(M\backslash \Gamma)$ and we have 
$$
\|u\|_{C^{2,\alpha}_{\gamma}(M\backslash \Gamma)}+\|u\|_{C^{1,\alpha}_{\gamma}(M)}\leq C\(\|\Delta_g u\|_{C^{0,\alpha}_{\gamma-2}(M\backslash\Gamma)}+\sum_{i=1}^{2}\|\d u/\d \eta_i\|_{C^{1,\alpha}_{\gamma-1}(\Sigma_i\backslash\Gamma)}+\|u\|_{C^{1}(K)}\).
$$ 
\end{lemma}
\begin{proof}
We  first identify $M$ outside a compact set with $\{x\in\mathbb R^n_{\mathbb L}\,;\:|x|>1\}$ by means of asymptotic coordinates $(x_1,...,x_n)$ and set $A=\{x\in M\,;\:1<|x|<4\}$ and $\widetilde A=\{x\in M\,;\:2<|x|<3\}$. Let $0\leq \chi\leq 1$ be a cut-off function satisfying $\chi\equiv 1$ in $\widetilde A$ and $\chi\equiv 0$ in $M\backslash A$. Assume without loss of generality that $\partial\chi/\partial \eta_{g}=0$. 

We set $u_R(x)=u(Rx)$ for $x\in A$ and define a metric $g_R$ on $A$ by 
$(g_R)_{ij}(x)=g_{ij}(Rx)$ in the asymptotic coordinates. 
It follows from elliptic regularity (similar to Proposition \ref{propo:exist:cpct}) that $u\in C_{loc}^{2,\alpha}(M\backslash\Gamma)\cap C_{loc}^{1,\alpha}(M)$ and 
$$
\|\chi u_R\|_{C^{2,\alpha}(A\backslash\Gamma)}+\|\chi u_R\|_{C^{1,\alpha}(A)}\leq
C\|\Delta_{g_R}(\chi u_R)\|_{C^{0,\alpha}(A\backslash\Gamma)}+C\|\partial (\chi u_R)/\partial\eta_{g_R}\|_{C^{1,\alpha}(A\cap((\Sigma_1\cup\Sigma_2)\backslash\Gamma))}
$$
for some $C=C(M,\Sigma_1,\Sigma_2,g)$. Observe that 
$$
\|u_R\|_{C^{1,\alpha}(\widetilde A)}\leq \|\chi u_R\|_{C^{1,\alpha}(A)}, \qquad
\|u_R\|_{C^{2,\alpha}(\widetilde A\backslash\Gamma)}\leq \|\chi u_R\|_{C^{2,\alpha}(A\backslash\Gamma)}, 
$$
$$
\|\Delta_{g_R}(\chi u_R)\|_{C^{0,\alpha}(A\backslash\Gamma)}\leq C\|\Delta_{g_R} u_R\|_{C^{0,\alpha}(A\backslash\Gamma)}+C\|u_R\|_{C^{1,\alpha}(A\backslash\Gamma)}
$$
and
$$
\|\partial (\chi u_R)/\partial\eta_{g_R}\|_{C^{1,\alpha}(A\cap((\Sigma_1\cup\Sigma_2)\backslash\Gamma)))}\leq C\|\partial u_R/\partial\eta_{g_R}\|_{C^{1,\alpha}(A\cap((\Sigma_1\cup\Sigma_2)\backslash\Gamma)))},
$$
so that 
\begin{align*}
\|u_R\|_{C^{2,\alpha}(\widetilde A\backslash\Gamma)}+\|u_R\|_{C^{1,\alpha}(\widetilde A)}
\leq
C\|\Delta_{g_R} u_R\|_{C^{0,\alpha}(A\backslash\Gamma)}&+C\|\partial u_R/\partial\eta_{g_R}\|_{C^{1,\alpha}(A\cap((\Sigma_1\cup\Sigma_2)\backslash\Gamma)))}
\\
&+C\|u_R\|_{C^{1,\alpha}(A\backslash\Gamma)}.
\end{align*}
Expanding this in terms of $C^0$ norms, multiplying by $R^{-\gamma}$ and rewriting the result in terms of $u$ and $g$, we get
\begin{align*}
\|u\|_{C_{\gamma}^{2,\alpha}(\widetilde A_R\backslash\Gamma)}+\|u\|_{C_{\gamma}^{1,\alpha}(\widetilde A_R)}
\leq
C\|\Delta_{g} u\|_{C_{\gamma-2}^{0,\alpha}(A_R\backslash\Gamma)}&+C\|\partial u/\partial\eta_{g}\|_{C_{\gamma-1}^{1,\alpha}(A_R\cap((\Sigma_1\cup\Sigma_2)\backslash\Gamma)))}
\\
&+C\|u\|_{C_{\gamma}^{1,\alpha}(A_R\backslash\Gamma)}.
\end{align*}
The proof now proceeds similarly to \cite[Lemma A.2]{almaraz-barbosa-lima}.
\end{proof}

\begin{lemma}\label{lemma:A3}
	Let $(M,\Sigma_1,\Sigma_2,g)$ be as in Lemma \ref{lemma:A2} and consider the operators $L=\Delta_g+h$ and $B_i=\d/\d\eta_i+h_i$, $i=1,2$, where $h\in C^{0,\alpha}_{-2-\epsilon}(M)$, $h_i\in C^{1,\alpha}_{-1-\epsilon}(\Sigma_i)$, for some $\epsilon>0$ small and $0<\alpha<1$. If $2-n<\gamma<0$, we define by $T(u)=(Lu, B_1u, B_2 u)$ the operator
	$$T:C^{1,\alpha}_{\gamma}(M)\cap C^{2,\alpha}_{\gamma}(M\backslash \Gamma)\rightarrow C^{0,\alpha}_{\gamma-2}(M\backslash \Gamma)\times C^{1,\alpha}_{\gamma-1}(\Sigma_1\backslash\Gamma)\times C^{1,\alpha}_{\gamma-1}(\Sigma_2\backslash\Gamma).$$  
	If $T$ is injective then there holds 
	\begin{equation}
		\|u\|_{C^{2,\alpha}_{\gamma}(M\backslash \Gamma)}+\|u\|_{C^{1,\a}_{\gamma}(M)}\leq C\(\|Lu\|_{C^{0,\alpha}_{\gamma-2}(M\backslash\Gamma)}+\sum_{i=1}^{2}\|B_iu\|_{C^{1,\alpha}_{\gamma-1}(\Sigma_i\backslash\Gamma)}\),
	\end{equation}
for all $u\in C^{1,\alpha}_{\gamma}(M)\cap C^{2,\alpha}_{\gamma}(M\backslash \Gamma)$ and some $C=C(M,g,\gamma,h,h_1, h_2)>0$.
\end{lemma}
\begin{proof}
The proof is similar to \cite[Lemma A.3]{almaraz-barbosa-lima} and we outline it here. As in the proof of Lemma \ref{lemma:A2}, we identify $M$ outside a compact set with $\{x\in\mathbb R^n_{\mathbb L}\,;\:|x|>1\}$ by means of asymptotic coordinates $(x_1,...,x_n)$ and consider the Euclidean metric $g_0$. 
Define a smooth function $r(x)$, $x\in M$, satisfying $r(x)=|x|$ for $|x|>1$, and $0< r(x)\leq 1$ otherwise. For $R\geq 1$, set 
$$
K_R=\{x\in M\,;\:r(x)\leq R\}
$$ 
and consider a smooth cut-off function $0\leq\theta\leq 1$ satisfying $\theta\equiv 1$ in $M\backslash K_{2}$  and $\theta\equiv 0$ in $K_1$. Since the support of $\theta$ is contained in $M\backslash K_1$, it makes sense to define $\theta_R$, 
for $R\geq 1$,  by $\theta_R(x)=\theta(R^{-1}x)$ which has support in $M\backslash K_R$.
Also, define the metric $g_R$ on $M$ by 
$$
g_R=\theta_Rg_0 +(1-\theta_R)g,
$$
so that $g_R=g_0$ in $M\backslash K_{2R}$. 

By the last assertion in  Lemma \ref{lemma:A2} there exists $C>0$ such that
\begin{align}\label{lemma3:2}
\|u\|_{C^{2,\a}_{\gamma}(M\backslash\Gamma)}+\|u\|_{C^{1,\a}_{\gamma}(M)}
\leq
C\|\Delta_{g_R}u\|_{C^{0,\a}_{\gamma-2}(M\backslash\Gamma)}+
C\|\d u/\d\eta_{g_R}\|_{C^{1,\a}_{\gamma-1}(M\backslash\Gamma)}
+C\|u\|_{C^{1}(K_{R'})}
\end{align}
for some large $R'\geq 2R$.
As in \cite[Lemma A.3]{almaraz-barbosa-lima} we  estimate 
\begin{align}\label{lemma3:5}
\|Lu-\Delta_{g_R}u\|_{C^{0,\a}_{\gamma-2}(M\backslash\Gamma)}
&\leq
C\|u\|_{C^{0,\a}(K_{2R}\backslash\Gamma)}+C(R^{-\tau}+R^{-\e})\|u\|_{C^{2,\a}_{\gamma}(M\backslash\Gamma)}
\end{align}
and
\begin{eqnarray}\label{lemma3:8}
\Big\|Bu-\frac{\d u}{\d\eta_{g_R}}\Big\|_{C^{1,\a}_{\gamma-1}((\Sigma_1\cup\Sigma_2)\backslash\Gamma)}
& \leq & 
C\|u\|_{C^{1,\a}(K_{2R}\cap((\Sigma_1\cup\Sigma_2)\backslash\Gamma))} +\\
& & \quad +C(R^{-\tau}+R^{-\e})\|u\|_{C^{2,\a}_{\gamma}((\Sigma_1\cup\Sigma_2)\backslash\Gamma)}\notag
\end{eqnarray}
so that (\ref{lemma3:2}), (\ref{lemma3:5}) and (\ref{lemma3:8}) lead to 
\begin{eqnarray*}
\|u\|_{C^{2,\a}_{\gamma}(M\backslash\Gamma)}+\|u\|_{C^{1,\a}_{\gamma}(M)}
&\leq  &
C\|Lu\|_{C^{0,\a}_{\gamma-2}(M\backslash\Gamma)}+C(R^{-\tau}
+R^{-\e})(\|u\|_{C^{2,\a}_{\gamma}(M\backslash\Gamma)}+\|u\|_{C^{1,\a}_{\gamma}(M)})
\\
& &\quad +C\|Bu\|_{C^{1,\a}_{\gamma-1}((\Sigma_1\cup\Sigma_2)\backslash\Gamma)}+C\|u\|_{C^{1,\a}(K_{R'})}.
\end{eqnarray*}
Hence, if we choose  $R$  large we obtain
\begin{align}\label{lemma3:9}
\|u\|_{C^{2,\a}_{\gamma}(M\backslash\Gamma)}+\|u\|_{C^{1,\a}_{\gamma}(M)}
&\leq
C\|Lu\|_{C^{0,\a}_{\gamma-2}(M\backslash\Gamma)}+C\|Bu\|_{C^{1,\a}_{\gamma-1}((\Sigma_1\cup\Sigma_2)\backslash\Gamma)}+C\|u\|_{C^{1,\a}(K_{R'})}\,.
\end{align}

The rest of the proof of Lemma \ref{lemma:A3} follows by a contradiction argument using the injectivity assumption and (\ref{lemma3:9}).
\end{proof}

\begin{lemma}\label{lemma:A4}
Let $(M,g)$ be an asymptotically flat manifold as in Lemma \ref{lemma:A3}. If $2-n<\gamma<0$ consider the operator 
	$$T:C^{1,\alpha}_{\gamma}(M)\cap C^{2,\alpha}_{\gamma}(M\backslash \Gamma)\rightarrow C^{0,\alpha}_{\gamma-2}(M\backslash \Gamma)\times C^{1,\alpha}_{\gamma-1}(\Sigma_1\backslash \Gamma)\times C^{1,\alpha}_{\gamma-1}(\Sigma_2\backslash \Gamma)$$  
	given by $T(u)=\(\Delta_g u,\d u/ \eta_1, \d u/ \eta_2\)$. If $g$ coincides with the Euclidean metric outside a compact subset of $M$ then $T$ is an isomorphism. 
\end{lemma}
\begin{proof}
The proof is similar to \cite[Lemma A.4]{almaraz-barbosa-lima} making use of Proposition \ref{propo:exist:cpct}.
\end{proof}

\begin{proof}[Proof of Proposition~\ref{prop:isom}]
First observe that all the operators $T$ as in the proposition are injective, as we can see by applying the maximum principle.
Indeed, due to the Neumann boundary condition, a doubling argument along $\Sigma_2$ allows one to use the Hopf lemma for the resulting manifold. 
Let $\widetilde{\mathcal{C}}$ be the set of all these operators and let $\mathcal{C}\subset\widetilde{\mathcal{C}}$ be the subset of isomorphisms. We consider $\widetilde{\mathcal{C}}$ with the operator norm  topology.

It follows from the Implicit Function Theorem that $\mathcal{C}$ is open in $\widetilde{\mathcal{C}}$. We will  prove that it is also closed.
We set $X=C^{0,\a}_{\gamma-2}(M\backslash \Gamma)$,  $Y_i=C^{1,\a}_{\gamma-1}(\Sigma_i\backslash \Gamma)$, $i=1,2$ and consider $X\times Y_1\times Y_2$ with the norm 
$$
\|(f,\bar{f_1},\bar{f_2})\|_{X\times Y_1\times Y_2}=\|f\|_X+\|\bar{f_1}\|_{Y_1}+\|\bar{f_2}\|_{Y_2}.
$$

Let $T_j\in \mathcal{C}$  be a sequence converging to some $T\in\widetilde{\mathcal{C}}$ under the operator norm $\|\,\|_{op}$. We shall prove that $T$ is surjective. 
Given $(f,\bar{f_1},\bar{f_2})\in X\times Y_1\times Y_2$ we must find $u\in C^{1,\alpha}_{\gamma}(M)\cap C^{2,\a}_{\gamma}(M\backslash\Gamma)$ such that $T(u)=(Lu,B_1u,B_2u)=(f,\bar{f_1},\bar{f_2})$. 

Let us write $T_j=(L_j,B_{1,j},B_{2,j})$. By hypothesis, there exists $u_j\in C^{1,\alpha}_{\gamma}(M)\cap C^{2,\a}_{\gamma}(M\backslash\Gamma)$ satisfying $(L_j u_j,B_{1,j}u_j,B_{2,j}u_j)=(f,\bar{f_1},\bar{f_2})$, so that, by Lemma \ref{lemma:A3}, there exists $C>0$ such that
$$
\|u_j\|_{ C^{2,\a}_{\gamma}(M\backslash\Gamma)}+\|u_j\|_{C^{1,\a}_{\gamma}(M)}\leq C\|(f,\bar{f_1},\bar{f_2})\|_{X\times Y_1\times Y_2},
$$
for all $j$.

In particular,  $u_j$ is uniformly bounded in $C^{1,\alpha}_{\gamma}(M)\cap C^{2,\a}_{\gamma}(M\backslash\Gamma)$. 
If we choose $\a_1\in (0,\a)$ and $\gamma_1\in(\gamma, 0)$, it follows from \cite[Lemma 3]{chaljub-choquet} that, by eventually passing to a subsequence, we may assume that $u_j\to u$  in $C^{1,\a_1}_{\gamma_1}(M)\cap C^{2,\a_1}_{\gamma_1}(M\backslash\Gamma)$, for some $u\in C^{1,\alpha}_{\gamma}(M)\cap C^{2,\a}_{\gamma}(M\backslash\Gamma)$. 
We just need to prove that $L_j u_j\to Lu$ in $C^0(M\backslash\Gamma)$ and  $B_{i,j} u_j\to B_iu$ in $C^0(\Sigma_i\backslash \Gamma)$, $i=1,2$, to conclude that $(Lu,B_1u,B_2u)=(f,\bar{f_1},\bar{f_2})$.

Observe that $\|T-T_j\|_{op}\to 0$ implies that $\|L-L_j\|_{op}\to 0$ and $\|B-B_{i,j}\|_{op}\to 0$, $i=1,2$. We also have
\begin{align}\label{propo1:1}
\|L_j u_j-Lu\|_{C^0(M\backslash\Gamma)}\leq \|L(u_j-u)\|_{C^0(M\backslash\Gamma)}+\|(L_j-L) u_j\|_{C^0(M\backslash\Gamma)}\,.
\end{align}
The first term on the right-hand side of (\ref{propo1:1}) converges to zero because $u_j\to u$ in $C^{1,\alpha_1}_{\gamma_1}(M)\cap C^{2,\a_1}_{\gamma_1}(M\backslash\Gamma)$. As for the  second one, 
$$
\|(L_j-L) u_j\|_{C^0(M\backslash\Gamma)}
\leq 
\|(L_j-L) u_j\|_{C^{0,\a}_{\gamma-2}(M\backslash\Gamma)}
\leq \|L_j-L\|_{op}\|u_j\|_{C^{1,\alpha}_{\gamma}(M)\cap C^{2,\a}_{\gamma}(M\backslash\Gamma)}\to 0\,,
$$
since $\|u_j\|_{C^{1,\alpha}_{\gamma}(M)\cap C^{2,\a}_{\gamma}(M\backslash\Gamma)}$ is uniformly bounded. This proves that $\|L_j u_j-Lu\|_{C^0(M\backslash\Gamma)}\to 0$. The proof that $\|B_{i,j} u_j-B_iu\|_{C^0(\Sigma_i\backslash \Gamma)}\to 0$, $i=1,2$, is similar, which  proves that $\mathcal{C}$ is closed in $\widetilde{\mathcal{C}}$.

Finally, we need to prove that $\widetilde{\mathcal{C}}$ is connected and contains an isomorphism. Using the notation in the proof of Lemma \ref{lemma:A3}, we consider the family of metrics $g_R$ for $R\geq 1$, and observe that the operators of the form $(-\Delta_{g_R}, \d/\d \eta_{1,g_R}, \d/\d \eta_{2,g_R})$ are isomorphisms, according to Lemma \ref{lemma:A4}. We set $\displaystyle L^t=-\Delta_{g_{(1-t)^{-1}}}+th$, $B_1^t=\d/\d\eta_{1,g_{(1-t)^{-1}}}+t\bar{h_1}$,  $B_2^t=\d/\d\eta_{2,g_{(1-t)^{-1}}}$, and define $T^t=(L^t,B^t_1,B^t_2)$ for $t\in [0,1)$ and $T^1=T$. Then $\{T^t\}_{t\in [0,1]}$ is a continuous family  of operators in $\widetilde{\mathcal{C}}$ connecting $(-\Delta_{g_1}, \d/\d \eta_{1,g_1},\d/\d \eta_{2,g_1})$ to $T$. This finishes the proof of Proposition \ref{prop:isom}.

%%%%%%%%%%%%%%%%%%%%%%%%%%%%%%%%%%%%%%%%

\end{proof}

%%%%%%%%%%%%%%%%%%%%%%%%%%%%%%%%%%%%%%%%%%%%%%%%%%%

\bigskip\noindent
{\bf{Acknowledgement.}} Part of this work was carried out during the first author's visit to Nanjing University of Science and Technology in the summer of 2024. He is very grateful to the hospitality of the School of Mathematics and Statistics. The authors would want to also thank Huiyan for the nice figures.

\bigskip\noindent
\textsc{S\'ergio Almaraz\\
Instituto de Matem\'atica e Estat\' istica, \\
Universidade Federal Fluminense\\
Rua Prof. Marcos Waldemar de Freitas S/N,
Niter\'oi, RJ,  24.210-201, Brazil}\\
e-mail: {\bf{sergioalmaraz@id.uff.br}}

\bigskip\noindent
\textsc{Shaodong Wang\\
School of Mathematics and Statistics,\\
Nanjing University of Science and Technology\\
Nanjing, 210094, People’s Republic of China} \\
e-mail: {\bf{shaodong.wang@mail.mcgill.ca}}

\end{document}